\documentclass[reqno,11pt]{article}

\usepackage{amsmath, amssymb}

\usepackage{amsmath}                    
\usepackage{amssymb}                    
\usepackage{amsthm}                     
\usepackage{graphicx}                   

\usepackage{textcomp}                   
\usepackage[T1]{fontenc}                
\usepackage{marvosym}                   
\usepackage[sc]{mathpazo}               

\usepackage{authblk}                    
\usepackage[usenames]{xcolor}           
\usepackage{lastpage}                   

\usepackage{enumerate}                  

\usepackage{hyperref}                   

\newtheorem{theorem}{Theorem}[section]
\newtheorem{corollary}[theorem]{Corollary}

\newtheorem{proposition}[theorem]{Proposition}

\theoremstyle{definition}

\theoremstyle{definition}
\newtheorem{remark}[theorem]{Remark}
\theoremstyle{definition}
\newtheorem{example}[theorem]{Example}

\numberwithin{equation}{section} \numberwithin{table}{section}
\numberwithin{figure}{section}

\title{\bf Viewing nonoscillatory second order linear differential equations from the angle of  Riccati equations}

\medskip
\author[1]{\textbf{Jaroslav Jaro\v s}\footnote{Corresponding author. Email: jaros@fmph.uniba.sk}}
\author[2]{\textbf{Kusano Taka\^si}}
\author[3] {\textbf{Tomoyuki Tanigawa}}

\affil[1]{Department of Mathematical Analysis and Numerical
	Mathematics, Faculty of Mathematics, Physics and Informatics,
	Comenius University, Mlynsk\'a dolina, Bratislava, 842 48,
	Slovakia} \affil[2]{Department of Mathematics, Faculty of Science,
	Hiroshima University, Higashi Hiroshima, 739-8526, Japan}

\affil[3]{ Department of Mathematical Sciences, Osaka Prefecture
	University, Osaka 599-8531, Japan, ttanigawa@ms.osakafu-u.ac.jp}

\begin{document}
	
	\maketitle

\pagestyle{plain}

\begin{center}
\noindent
\begin{minipage}{0.85\textwidth}\parindent=15.5pt

{\small{ \noindent {\bf Abstract.} We build an existence theory for second order linear differential
equations of the form
\begin{equation*}
(p(t)x')'=q(t)x, \eqno{\textrm{(A)}}
\end{equation*}
$p(t)$ and $q(t)$ being positive continuous functions on
$[a,\infty)$, in which a crucial role is played by a pair of the
Riccati differential equations
\begin{equation*}
u'=q(t)-\frac{u^2}{p(t)} \qquad \qquad {\rm (R1)}\ ,\qquad \qquad
\qquad v'=\frac{1}{p(t)}-q(t)v^2 \quad \qquad \qquad {\rm (R2)}
\end{equation*}
associated with (A). An essential part of the theory is the
construction of a pair of linearly independent nonoscillatory
solutions $x_1(t)$ and $x_2(t)$ of (A) enjoying explicit
exponential-integral representations in terms of solutions
$u_1(t)$ and $u_2(t)$ of (R1) or in terms of solutions $v_1(t)$
and $v_2(t)$ of (R2). }
\smallskip

\smallskip

}

\end{minipage}
\end{center}

\section{Introduction}

\indent

Consider the second order linear differential equation
\begin{equation*}
(p(t)x')'=q(t)x, \eqno{\textrm{(A)}}
\end{equation*}
where $p: [a,\infty)\to(0,\infty)$ and $q: [a,\infty)\to{\bf R}$,
$a\geq 0$, are continuous functions.

We are concerned exclusively with nontrivial solutions of (A)
defined in some neighborhood of infinity, that is, on an interval
of the form $[T,\infty)$, $T\geq a$. Such a solution is called
{\it oscillatory} if it has an infinite sequence of zeros
clustering at infinity, and {\it nonoscillatory} otherwise.

It is known that (A) does not admit both oscillatory solutions and
nonoscillatory solutions simultaneously, that is to say, all
solutions of (A) are either oscillatory, in which case (A) is said
to be oscillatory, or else nonoscillatory, in which case (A) is
said to be nonoscillatory. For example, if $q(t)>0$ on
$[a,\infty)$, then equation (A) is always nonoscillatory, whereas
if $q(t)<0$ on $[a,\infty)$, then it is possible that (A) is
either oscillatory or nonoscillatory depending on the delicate
interrelations between $p(t)$ and $q(t)$. In order to detect
criteria for oscillation of equation (A) with $q(t)<0$ frequent
use has been made of the first order nonlinear differential
equation
\begin{equation*}
u'+\frac{u^2}{p(t)}-q(t)=0, \eqno{\textrm{(R1)}}
\end{equation*}
which is referred to as the Riccati differential equation
associated with (A). This (R1) is derived as a differential
equation satisfied by the function $u(t)=p(t)x'(t)/x(t)$ for any
nonoscillatory solution $x(t)$ of (A) defined in a neighborhood of
infinity. A solution of (R1) existing for all large $t$ is named a
{\it global solution} of (R1). Conversely, if $u(t)$ is a global
solution of (R1) on $[T,\infty)$, then $x(t)=\exp(\int_T^t
u(s)/p(s)ds)$ gives a solution of (A) on that interval. Thus there
exists a noteworthy relationship between (A) and (R1) as described
in the following
\begin{proposition}
\; {\it Equation} (A) is {\it nonoscillatory if and only if
equation} (R1) {\it has a global solution.}
\end{proposition}
On the basis of this fact a wealth of oscillation criteria for
equation (A) have been obtained in the literature; see e.g.,
Swanson [9]. Those oscillation criteria for (A) represent one side
of the usefulness of Proposition 1.1. However, there seems to be
almost no past study shedding a light on another side of its
power, namely the possibility of a detailed qualitative study of
nonoscillatory solutions of (A) with the aid of global solutions
of (R1). A result (the only one we know of) in this regard can be
found in the book of Hille [5, Theorem 9.4.3] concerning the
special case of (A) with $p(t)\equiv 1$ and $q(t)<0$ on
$[a,\infty)$.

In this paper an attempt is made to develop a new nonoscillation
theory for the purely nonoscillatory equation (A) with $q(t)>0$ by
means of the Riccati equation (R1) and its equivalent
\begin{equation*}
v'=\frac{1}{p(t)}-q(t)v^2.   \eqno{\textrm{(R2)}}
\end{equation*}
Note that (R2) was recently found by Mirzov [7] as the equation
satisfied by $v(t)=x(t)/p(t)x'(t)$, where $x(t)$ is any
nonoscillatory solution of (A).

To specify what we are going to do with (A) and (R1) - (R2) we
need to recollect some basic facts about the solution set
$\mathcal{S}(\textrm{A})$ of equation (A).

(1)\; $\mathcal{S}(\textrm{A})\neq \emptyset$ and all of its
members are nonoscillatory.

(2)\; $\mathcal{S}(\textrm{A})$ is a two-dimensional linear space
over the reals {\bf R}. As a basis for $\mathcal{S}(\textrm{A})$
one can take a pair of positive solutions $\{x_1(t),x_2(t)\}$ of
(A) satisfying
\begin{equation}\label{1.1}
\lim_{t\to\infty}\frac{x_1(t)}{x_2(t)}=0,
\end{equation}
or
\begin{equation}\label{1.2}
\int_{}^{\infty}\frac{dt}{p(t)x_1(t)^2}=\infty \;\;\;
\textrm{and}\;\;\; \int_{}^{\infty}\frac{dt}{p(t)x_2(t)^2}<\infty.
\end{equation}
See Hartman [4, Theorem 6.4; Corollary 6.4]. The functions
$x_1(t)$ and $x_2(t)$ are called, respectively, a {\it principal
solution} and a {\it nonprincipal solution} of (A). A principal
solution is uniquely determined up to a constant factor. See [4,
Theorem 6.4]. It may be assumed that $x_1(t)$ is decreasing and
$x_2(t)$ is increasing; see [4, Corollary 6.4].

(3)\; If $x(t)$ is any solution of (A) on $[T,\infty)$, then
\begin{equation}\label{1.3}
x_2(t)=x(t)\int_T^t \frac{ds}{p(s)x(s)^{2}}
\end{equation}
is a nonprincipal solution on $[T,\infty)$. If, in addition,
$x(t)$ is a nonprincipal solution of (A), then
\begin{equation}\label{1.4}
x_1(t)=x(t)\int_t^{\infty} \frac{ds}{p(s)x(s)^{2}}
\end{equation}
is a principal solution of (A) on $[T,\infty)$. See [4, Corollary
6.3].

(4)\; Information on the possible asymptotic behavior of $x_1(t)$
and $x_2(t)$ forming a basis for $\mathcal{S}(\textrm{A})$
follows. We need to discuss the three cases $\{I_p=\infty\}$,
$\{I_q=\infty\}$ and $\{I_p<\infty\wedge I_q<\infty\}$ separately,
where
\begin{equation}\label{1.5}
I_p=\int_a^{\infty}\frac{dt}{p(t)}\quad\textrm{and}\quad
I_q=\int_a^{\infty}q(t)dt.
\end{equation}

(i)\; Suppose that $I_p=\infty$. Then $x_1(t)$ and $x_2(t)$
satisfy either
\begin{equation}\label{1.6}
\lim_{t\to\infty}x_1(t)=\textrm{const}>0\;\;\;\textrm{and}\;\;\;
\lim_{t\to\infty}\frac{x_2(t)}{P(t)}=\textrm{const}>0,
\end{equation}
or
\begin{equation}\label{1.7}
\lim_{t\to\infty}x_1(t)=0\;\;\;\textrm{and}\;\;\;
\lim_{t\to\infty}\frac{x_2(t)}{P(t)}=\infty,
\end{equation}
where $P(t)$ is defined by
\begin{equation}\label{1.8}
P(t)=\int_a^t \frac{ds}{p(s)}.
\end{equation}

(ii)\; Suppose that $I_q=\infty$. Then $x_1(t)$ and $x_2(t)$
satisfy either
\begin{equation}\label{1.9}
\lim_{t\to\infty}\frac{x_1(t)}{\pi(t)}=\textrm{const}>0\;\;\;\textrm{and}\;\;\;
\lim_{t\to\infty}x_2(t)=\textrm{const}>0,
\end{equation}
or
\begin{equation}\label{1.10}
\lim_{t\to\infty}\frac{x_1(t)}{\pi(t)}=0\;\;\;\textrm{and}\;\;\;
\lim_{t\to\infty}x_2(t)=\infty,
\end{equation}
where $\pi(t)$ is given by
\begin{equation}\label{1.11}
\pi(t)=\int_t^{\infty}\frac{ds}{p(s)}.
\end{equation}

(iii)\; Suppose that $I_p<\infty\wedge I_q<\infty$. Then $x_1(t)$
and $x_2(t)$ satisfy
\begin{equation}\label{1.12}
\lim_{t\to\infty}\frac{x_1(t)}{\pi(t)}=\textrm{const}>0\;\;\;\textrm{and}\;\;\;
\lim_{t\to\infty}x_2(t)=\textrm{const}>0.
\end{equation}

Solutions $x_1(t)$ and $x_2(t)$ satisfying \eqref{1.6},
\eqref{1.9} and \eqref{1.12} are called {\it moderate solutions}
of (A), while those satisfying \eqref{1.7} and \eqref{1.10} are
called {\it extreme solutions} of (A). A pair of solutions
$\{x_1(t),x_2(t)\}$ is termed a {\it moderate basis} or an {\it
extreme basis} for $\mathcal{S}${\textrm{(A)} according to whether
both $x_1(t)$ and $x_2(t)$ are moderate solutions or extreme
solutions of (A). It will turn out later that no basis for
$\mathcal{S}${\textrm{(A)} is allowed to be a combination of a
moderate solution and an extreme solution.

Is it possible to establish criteria for judging the existence of
a moderate basis or an extreme basis for equation (A)?  Answering
this question regarding moderate bases is relatively easy, and as
is described in Section 3 one can find necessary and sufficient
conditions for (A) to possess a moderate basis whose components
$x_1(t)$ and $x_2(t)$ satisfy \eqref{1.6}, \eqref{1.9} or
\eqref{1.12}. It is demonstrated that all moderate bases can be
constructed by way of the Riccati equations (R1) and (R2), both of
which are indispensable in the construction process. Although the
Riccati differential equations have a long history (see e.g.,
Sansone [8]), their potential for reproduction of solutions of the
associated second order linear differential equations has not been
recognized until now. We note that some results of Section 3 are
known (see e.g.,[1], [2], [6]), but the derivation is essentially
different.

What about extreme bases for (A)?  This question seems to be
extremely difficult to answer for now. As far as we are aware,
historically no systematic study has been made of the existence
and asymptotic behavior of extreme solutions of (A). However, it
is possible to detect several nontrivial cases in which extreme
solutions are produced by effective application of (R1) and (R2).
The details of how to develop a tiny theory of extreme solutions
of (A) are provided in Section 4. Examples illustrating all of our
theorems are given in Section 5.

What is requisite in building our theory for (A) are an a priori
knowledge of the overall structure of solutions of (A), or
equivalently, the classification of $\mathcal{S}\textrm{(A)}$ into
appropriate disjoint subclasses according to the asymptotic
behavior at infinity of its members, and understanding of the
Riccati equations (R1) and (R2) as an instrument for reproducing
solutions of (A) from their solutions. All these things together
are explained in the preparatory Section 2.

\section{Preparatory observations}

\subsection{Classification of nonoscillatory solutions of (A)}

\indent

Let $x(t)$ be a (nonoscillatory) solution of equation (A). There
exists $T\geq a$ such that $x(t)Dx(t)\neq 0$ (i.e.,$x(t)x'(t)\neq
0$) for $t\geq T$. Since (A) implies that both $x(t)$ and $Dx(t)$
are monotone on $[T,\infty)$, they have the limits
$x(\infty)=\lim_{t\to\infty}x(t)$ and
$Dx(\infty)=\lim_{t\to\infty}Dx(t)$ in the extended real number
system. The pair $(x(\infty),Dx(\infty))$ is a crucial indicator
of the asymptotic behavior of a solution $x(t)$ of (A) as
$t\to\infty$. It is referred to as the {\it terminal state} of
$x(t)$. Our purpose here is to show that one can enumerate in
advance all possible patterns of terminal states of solutions of
(A) by considering the three cases $I_p=\infty$, $I_q=\infty$ and
$I_p<\infty\wedge I_q<\infty$ separately (cf. \eqref{1.5}).

(I)\; We start with the case $I_p=\infty$. Let $x(t)$ be a
solution of (A) such that $x(t)Dx(t)>0$ on $[T,\infty)$. By (A)
this means that $|x(t)|$ and $|Dx(t)|$ are increasing for $t\geq
T$. The limit $|Dx(\infty)|$ may or may not be finite. In either
case we have $|x(\infty)|=\infty$. In fact, if $x(t)$ is positive,
then from the inequality
$$
Dx(t)=p(t)x'(t)\geq Dx(T)\quad \textrm{or}\quad x'(t)\geq
\frac{Dx(T)}{p(t)}, \quad t\geq T,
$$
we see that $x(\infty)=\infty$. If $x(t)$ is negative, repeating
the same argument with $x(t)$ replaced by $-x(t)$, we are led to
the conclusion that $|x(\infty)|=\infty$.

Let $x(t)$ be a solution of (A) satisfying $x(t)Dx(t)<0$ on
$[T,\infty)$. Then, by (A) $|x(t)|$ is decreasing and $|Dx(t)|$ is
increasing for $t\geq T$. The limit $|x(\infty)|$ may or may not
be zero. In either case we have $|Dx(\infty)|=0$. Suppose for
contradiction that $|Dx(\infty)|=k>0$. If $x(t)$ is positive, then
since $Dx(t)$ is negative, we obtain
$$
Dx(t)=p(t)x'(t)\leq -k\quad \textrm{or}\quad x'(t)\leq
\frac{-k}{p(t)},\quad t\geq T,
$$
which implies $x(\infty)=-\infty$, an impossibility. If $x(t)$ is
negative, it suffices to repeat the same argument with $x(t)$
replaced by $-x(t)$.

From the above discussions it is concluded that for a solution
$x(t)$ of (A) satisfying $x(t)Dx(t)>0$ for all large $t$ the
following two types of its terminal state are possible
$$
|x(\infty)|=\infty\quad \textrm{and}\quad 0<|Dx(\infty)|<\infty,
\leqno{\textrm{I(i)}}
$$
$$
|x(\infty)|=\infty\quad \textrm{and}\quad |Dx(\infty)|=\infty,
\leqno{\textrm{I(ii)}}
$$
and that for a solution $x(t)$ of (A) satisfying $x(t)Dx(t)<0$ for
all large $t$ the following two types of terminal states are
possible
$$
0<|x(\infty)|<\infty\quad \textrm{and}\quad |Dx(\infty)|=0,
\leqno{\textrm{I(iii)}}
$$
$$
|x(\infty)|=0\quad \textrm{and}\quad |Dx(\infty)|=0.
\leqno{\textrm{I(iv)}}
$$

(II)\; We turn to the case $I_q=\infty$. Let $x(t)$ be any
solution of (A). If $x(t)Dx(t)>0$ on $[T,\infty)$, then $|x(t)|$
and $|Dx(t)|$ are increasing for $t\geq T$. The limit
$|x(\infty)|$ may be finite or infinite, but in either case we
have $|Dx(\infty)|=\infty$. In fact, if $x(t)$ is positive, then
integrating (A) from $T$ to $t$ gives
$$
Dx(t)=Dx(T)+\int_T^t q(s)x(s)ds \geq x(T)\int_T^t
q(s)ds\to\infty,\;\;\; t\to\infty,
$$
which implies that $Dx(\infty)=\infty$. The same argument as above
applies to $-x(t)$ leads us to $Dx(\infty)=-\infty$.

If $x(t)Dx(t)<0$ on $[T,\infty)$, then $|x(t)|$ is decreasing and
$|Dx(t)|$ is increasing. The limit $|Dx(\infty)|$ may be finite or
infinite, but in either case we have $|x(\infty)|=0$. Suppose to
the contrary that $|x(\infty)|=c>0$. If $x(t)$ is positive, then
integrating (A) on $[T,t]$, we obtain
$$
Dx(t)=Dx(T)+\int_T^t q(s)x(s)ds\geq Dx(T)+c\int_T^t
q(s)ds\to\infty,\;\;\; t\to\infty,
$$
which implies $Dx(\infty)=\infty$. But this contradicts the
negativity of $Dx(t)$. Similarly, if $x(t)$ is negative, we are
led to the contradiction that $Dx(\infty)\to-\infty$ as
$t\to\infty$.

Summarizing what is said above, we conclude that for a solution
$x(t)$ of (A) such that $x(t)Dx(t)>0$ for all large $t$ the
following two types of its terminal state are possible
$$
0<|x(\infty)|<\infty\quad \textrm{and}\quad |Dx(\infty)|=\infty,
\leqno{\textrm{II(i)}}
$$
$$
|x(\infty)|=\infty\quad \textrm{and}\quad |Dx(\infty)|=\infty,
\leqno{\textrm{II(ii)}}
$$
and that a solution $x(t)$ of (A) such that $x(t)Dx(t)<0$ for all
large $t$ the following two types of its terminal state are
possible
$$
|x(\infty)|=0\quad \textrm{and}\quad 0<|Dx(\infty)|<\infty,
\leqno{\textrm{II(iii)}}
$$
$$
|x(\infty)|=0\quad \textrm{and}\quad |Dx(\infty)|=0.
\leqno{\textrm{II(iv)}}
$$

(III)\; Finally we are concerned with the case $I_p<\infty\wedge
I_q<\infty$. It can be shown that if $x(t)$ is a solution of (A)
such that $x(t)Dx(t)>0$ for all large $t$, then its terminal state
satisfies
$$
0<|x(\infty)|<\infty \quad \textrm{and}\quad
0<|Dx(\infty)|<\infty, \leqno{\textrm{III(i)}}
$$
and that if $x(t)$ is a solution of (A) such that $x(t)Dx(t)<0$
for all large $t$, then its terminal state satisfies either
$$
|x(\infty)|=0\quad \textrm{and}\quad 0<|Dx(\infty)|<\infty,
\leqno{\textrm{III(ii)}}
$$
or
$$
0<|x(\infty)|<\infty \quad \textrm{and}\quad |Dx(\infty)|=0.
\leqno{\textrm{III(iii)}}
$$

It suffices to verify these facts for eventually positive
solutions of (A). Let $x(t)$ be a solution of (A) such that
$x(t)>0$ and $Dx(t)>0$ on $[T,\infty)$, where $T>a$ is chosen so
that
\begin{equation}\label{2.1}
\int_T^{\infty}\frac{1}{p(s)}\int_T^s q(r)drds\leq \frac{1}{2}.
\end{equation}
Integrating (A) twice on $[T,t]$, we have
\begin{equation}\label{2.2}
Dx(t)=Dx(T)+\int_T^t q(s)x(s)ds,
\end{equation}
\begin{equation}\label{2.3}
x(t)=x(T)+\int_T^t \frac{Dx(T)}{p(s)}ds +\int_T^t
\frac{1}{p(s)}\int_T^s q(r)x(r)drds,
\end{equation}
for $t\geq T$. From \eqref{2.1} and \eqref{2.3} we see that
$$
\frac{1}{2}x(t)\leq x(T)+Dx(T)\int_T^{\infty}\frac{ds}{p(s)},\quad
t\geq T,
$$
which implies that $0<x(\infty)<\infty$. This combined with
\eqref{2.2} gives
$$
Dx(t)\leq Dx(T)+x(t)\int_T^t q(s)ds\leq
Dx(T)+x(\infty)\int_T^{\infty}q(s)ds,\quad t\geq T,
$$
which implies that $0<Dx(\infty)<\infty$. Thus the terminal state
of $x(t)$ is of the type III(i).

Let $x(t)$ be a solution of (A) such that $x(t)>0$ and $Dx(t)<0$
for all large $t$. It is clear that $0\leq x(\infty)<\infty$ and
$-\infty<Dx(\infty)\leq 0$. But the possibility
$x(\infty)=Dx(\infty)=0$ should be excluded because if this would
occur, integrating (A) twice from $t$ to $\infty$ would give
$$
x(t)=\int_t^{\infty}\frac{1}{p(s)}\int_s^{\infty}q(r)x(r)drds \leq
x(t)\int_t^{\infty}\frac{1}{p(s)}\int_s^{\infty}q(r)drds,
$$
or
$$
1\leq \int_t^{\infty}\frac{1}{p(s)}\int_s^{\infty}q(r)drds\to
0,\quad t\to\infty,
$$
a contradiction. Therefore the terminal state of $x(t)$ must be
either of the type III(ii) or of the type III(iii).

Solutions of (A) having the terminal states of the types I(i),
I(iii), II(i), II(iii), III(i) - III(iii) are named {\it moderate
solutions} of (A), whereas those having the terminal states of the
types I(ii), I(iv), II(ii), II(iv) are named {\it extreme
solutions} of (A).

\begin{remark}\; (i)\; Suppose that $I_p=\infty\wedge
I_q=\infty$. Since both (I) and (II) apply to (A), we see that the
only possible types of terminal states of its solutions are I(ii)
(= II(ii)) and I(iv) (=II(iv)). In other words all solutions of
(A) must be extreme solutions.

(ii)\; Suppose that $I_p<\infty\wedge I_q<\infty$. In this case
all solutions of (A) are moderate solutions which are bounded
together with their quasi-derivatives for all large $t$.
\end{remark}

\subsection{Riccati equations associated with (A)}

\indent

Let $x(t)$ be any nonoscillatory solution of equation (A). Then
$x(t)x'(t)\neq 0$ on $[T,\infty)$ for some $T\geq a$. Define
\begin{equation}\label{2.4}
u(t)=\frac{p(t)x'(t)}{x(t)}\;\;\;\textrm{and}\;\;\;v(t)=\frac{x(t)}{p(t)x'(t)}.
\end{equation}
It is elementary to show that $u(t)$ and $v(t)$ satisfy,
respectively, the first order nonlinear differential equations
(R1) and (R2) on $[T,\infty)$ which are referred to as the first
and second Riccati equations associated with (A).

Conversely, suppose that (R1) has a global solution $u(t)$ on
$[T,\infty)$. Bearing \eqref{2.1} in mind define $x(t)$ by
\begin{equation}\label{2.5}
x(t)=\exp\Bigl(\int_T^t
\frac{u(s)}{p(s)}ds\Bigr)\;\;\;\textrm{or}\;\;\;
x(t)=\exp\Bigl(-\int_t^{\infty}\frac{u(s)}{p(s)}ds\Bigr),\quad
t\geq T,
\end{equation}
according to whether $\int_T^{\infty} u(t)/p(t)dt$ is divergent or
convergent. It is obvious that \eqref{2.5} gives a nonoscillatory
solution of (A) defined for $t\geq T$. Thus the truth of
Proposition 1.1 is verified. We may well call \eqref{2.5} a {\it
reproducing formula} for solutions of (A) by way of Riccati
equation (R1), or from solutions of (R1).

We note that an alternative reproducing formula for (A) by way of
(R1)
\begin{equation}\label{2.6}
x(t)=\frac{1}{u(t)}\exp\Bigl(\int_T^t
\frac{q(s)}{u(s)}ds\Bigr)\;\;\;\textrm{or}\;\;\;
x(t)=\frac{1}{u(t)}\exp\Bigl(-\int_t^{\infty}
\frac{q(s)}{u(s)}ds\Bigr), \quad t\geq T,
\end{equation}
can be derived without difficulty. It suffices to obtain
\begin{equation}\label{2.7}
\frac{(p(t)x'(t))'}{p(t)x'(t)}=\frac{q(t)}{u(t)},\quad t\geq T,
\end{equation}
by combining \eqref{2.1} rewritten as $x(t)=(p(t)x'(t))u(t)$ with
(A) rewritten as $x(t)=(p(t)x'(t))'/q(t)$, and then to integrate
\eqref{2.7} over $[T,t]$ or $(t,\infty)$ according as $q(t)/u(t)$
is non-integrable or integrable on $[T,\infty)$.

Similarly, reproducing solutions of (A) by way of the second
Riccati equation (R2) can be carried out via the following
formulas corresponding to \eqref{2.5} and \eqref{2.6}:
\begin{equation}\label{2.8}
x(t)=\exp\Bigl(\int_T^t
\frac{ds}{p(s)v(s)}\Bigr)\;\;\;\textrm{or}\;\;\;
x(t)=\exp\Bigl(-\int_t^{\infty} \frac{ds}{p(s)v(s)}\Bigr),\quad
t\geq T,
\end{equation}
and
\begin{equation}\label{2.9}
x(t)=v(t)\exp\Bigl(\int_T^t
q(s)v(s)ds\Bigr)\;\;\;\textrm{or}\;\;\;
x(t)=v(t)\exp\Bigl(-\int_t^{\infty} q(s)v(s)ds\Bigr),\quad t\geq
T.
\end{equation}

One of the central problems studied in this paper is the
reproduction of a pair of linearly independent solutions of
equation (A) from suitable global solutions of the Riccati
equations (R1) and/or (R2). Let $\{x_1(t),x_2(t)\}$ be a pair of
linearly independent solutions of (A) such that $x_i(t)x_i'(t)\neq
0$ for $t\geq T$, $i=1,2$. From the well-known identity
$$
p(t)[x_2(t)x_1'(t)-x_1(t)x_2'(t)]= C\neq 0 \;\;\;\textrm{(a
constant)},\quad t\geq T,
$$
it follows that
$$
\frac{p(t)x_1'(t)}{x_1(t)}-\frac{p(t)x_2'(t)}{x_2(t)}=\frac{C}{x_1(t)x_2(t)},
$$
and
$$\frac{x_1(t)}{p(t)x_1'(t)}-\frac{x_2(t)}{p(t)x_2'(t)}=
\frac{C}{p(t)^2x_1'(t)x_2'(t)},
$$
for $t\geq T$. This implies that if $x_i(t)$, $i=1,2$, are
reproduced from a pair of solutions $u_i(t)$, $i=1,2$, of (R1),
then $u_1(t)\neq u_2(t)$ for all $t\geq T$, and that the same is
true of a pair of solutions $v_i(t)$, $i=1,2$, of (R2) which
reproduce $x_i(t)$, $i=1,2$.

\begin{remark}\; Since a solution $x(t)$ of (A) given by
\eqref{2.5} or \eqref{2.6} satisfies $x(t)Dx(t)=u(t)x(t)^2$, it is
a principal solution or a nonprincipal solution according as
$u(t)$ is eventially positive or negative. Similarly, a solution
$x(t)$ of (A) given by \eqref{2.8} or \eqref{2.9} satisfies
$x(t)Dx(t)=x(t)^2/v(t)$, it is a principal solution or a
nonprincipal solution according as $v(t)$ is eventually positive
or negative.
\end{remark}

\section{Moderate solutions of (A)}

\indent

The main purpose of this section is to demonstrate that all
moderate solutions of equation (A) can be reproduced from global
solutions of the associated Riccati equations (R1) and (R2).

We first consider (A) with $p(t)$ satisfying $I_p=\infty$. Assume
that (A) has a moderate solution $x(t)$ such that $x(t)Dx(t)\neq
0$ for $t\geq T$. We may suppose that $x(t)$ is positive on
$[T,\infty)$. Let $x_1(t)$ be a solution of the type I(i). Then it
satisfies
\begin{equation}\label{equation3.1}
\lim_{t\to\infty}\frac{x_1(t)}{P(t)}=\lim_{t\to\infty}Dx_1(t)=d,
\end{equation}
for some $d>0$. Integrating the second equation of (A) from $t$ to
$\infty$, we have
\begin{equation*}
Dx_1(t)=d-\int_t^{\infty}q(s)x_1(s)ds,\quad t\geq T,
\end{equation*}
which means the integrability of $q(t)x_1(t)$ on $[T,\infty)$.
This fact combined with \eqref{equation3.1} gives
\begin{equation}\label{equation3.2}
\int_a^{\infty}P(t)q(t)dt<\infty.
\end{equation}

Let $x_2(t)$ be a solution of class I(iii) of (A). It clearly
satisfies
\begin{equation}\label{3.3}
\lim_{t\to\infty}x_2(t)=c \quad \textrm{and}\quad
\lim_{t\to\infty}\frac{Dx_2(t)}{\rho(t)}=-c,
\end{equation}
for some $c>0$, where $\rho(t)$ is defined by
\begin{equation}\label{3.4}
\rho(t)=\int_t^{\infty}q(s)ds,\quad t\geq a.
\end{equation}
Integrating (A) twice on $[t,\infty)$, we obtain
$$
x_2(t)=c+\int_t^{\infty}(P(s)-P(t))q(s)x_2(s)ds, \quad t\geq T,
$$
from which it follows that \eqref{equation3.2} must also be
satisfied. Thus we see that \eqref{equation3.2} is a necessary
condition for the existence of moderate solutions for (A)

A useful equivalent of \eqref{equation3.2} is given in the
following remark.

\begin{remark}\label{remark3.1}\; {\it If $I_p=\infty$, then}
\begin{equation}\label{3.5}
\int_a^{\infty}q(t)P(t)dt<\infty\quad \Longleftrightarrow \quad
\int_a^{\infty}\frac{\rho(t)}{p(t)}dt<\infty.
\end{equation}
In fact, by combining
$$
\int_a^t q(s)P(s)ds+\rho(t)P(t)=\int_a^t
\frac{\rho(s)}{p(s)}ds,\quad t\geq a,
$$
with $\rho(t)P(t)\leq \int_t^{\infty}q(s)P(s)ds$, we find that
$\int_a^{\infty} q(s)P(s)ds=\int_a^{\infty}\rho(s)/p(s)ds$.
\end{remark}

It turns out that \eqref{equation3.2} is also a sufficient
condition for the existence of moderate solutions of the types
I(i) and I(iii) of (A). Suppose that \eqref{equation3.2} and hence
\eqref{3.5} holds.

Choose $T>a$ so that
\begin{equation}\label{3.6}
\int_T^{\infty}\frac{\rho(s)}{p(s)}ds\leq \frac{1}{4},
\end{equation}
and denote by $\mathcal{U}$ the set of functions
\begin{equation}\label{3.7}
\mathcal{U}=\big\{u\in C_0[T,\infty): -\rho(t)\leq u(t)\leq
-\frac{1}{2}\rho(t),\;\; t\geq T\big\},
\end{equation}
where $C_0[T,\infty)$ is the totality of continuous functions on
$[T,\infty)$ tending to zero as $t\to\infty$. It is a Banach space
with the norm $\|u\|_0=\sup\{|u(t)|:t\geq T\}$. Define the
integral operator $F$ by
\begin{equation}\label{3.8}
Fu(t)=-\rho(t)+ \int_t^{\infty}\frac{u(s)^2}{p(s)}ds,\quad t\geq
T,
\end{equation}
and let it act on $\mathcal{U}$ which is a closed subset of
$C_0[T,\infty)$.

If $u\in \mathcal{U}$, then, since
$$
\int_t^{\infty}\frac{u(s)^2}{p(s)}ds\leq
\int_t^{\infty}\frac{\rho(s)^2}{p(s)}ds \leq
\rho(t)\int_t^{\infty}\frac{\rho(s)}{p(s)}ds\leq
\frac{1}{4}\rho(t), \quad t\geq T,
$$
we obtain $-\rho(t)\leq Fu(t)\leq -\rho(t)/2$ for $t\geq T$. This
shows that $F$ is a self-map of $\mathcal{U}$. If $u_1,\;u_2\in
\mathcal{U}$, then
$$
|Fu_1(t)-Fu_2(t)|\leq
\int_t^{\infty}\frac{1}{p(s)}|u_1(s)^2-u_2(s)^2|ds
$$
$$
\leq \int_t^{\infty}\frac{2\rho(s)}{p(s)}|u_1(s)-u_2(s)|ds \leq
2\int_t^{\infty}\frac{\rho(s)}{p(s)}ds{\cdot}\|u_1-u_2\|_0\leq
\frac{1}{2}\|u_1-u_2\|_0,
$$
from which it follows that
$$
\|Fu_1-Fu_2\|_0\leq \frac{1}{2}\|u_1-u_2\|_0.
$$
This means that $F$ is a contraction on $\mathcal{U}$. Therefore,
there exists a unique fixed point $u\in \mathcal{U}$ which
satisfies
\begin{equation}\label{3.9}
u(t)=-\rho(t)+ \int_t^{\infty}\frac{u(s)^2}{p(s)}ds,\quad t\geq T,
\end{equation}
and hence gives a global (negative) solution of (R1) on
$[T,\infty)$. Note that $-u(t)\sim \rho(t)$ as $t\to\infty$. Here
the symbol $\sim$ is used to mean the asymptotic equivalence of
two positive functions $f(t)$ and $g(t)$;
$$
f(t)\sim g(t),\quad t\to\infty \quad \Longleftrightarrow \quad
\lim_{t\to\infty}\frac{f(t)}{g(t)}=1.
$$
Since $u(t)/p(t)$ is integrable on $[T,\infty)$, we can use the
reproducing formula \eqref{2.5} to define a positive solution of
(A) by
\begin{equation}\label{3.10}
x_1(t)=\exp\Bigl(-\int_t^{\infty}\frac{u(s)}{p(s)}ds\Bigr),\quad
t\geq T.
\end{equation}
It is clear that $x_1(\infty)=1$. The quasi-derivative of $x_1(t)$
is given by
\begin{equation}\label{3.11}
Dx_1(t)=u(t)\exp\Bigl(-\int_t^{\infty}\frac{u(s)}{p(s)}ds\Bigr),\quad
t\geq T,
\end{equation}
and satisfies $Dx_1(t)\sim -\rho(t)$ as $t\to\infty$. Thus
$x_1(t)$ is a positive decreasing solution of the type I(iii),and
hence a principal solution of equation (A).

Moderate solutions of the type I(i) of (A) will be reproduced
under condition \eqref{equation3.2} by way of the second Riccati
equation (R2). Choose $T>a$ so that
\begin{equation}\label{3.12}
\int_T^{\infty}q(s)P(s)ds\leq \frac{1}{4},
\end{equation}
and define the set $\mathcal{V}$ by
\begin{equation}\label{3.13}
\mathcal{V}=\big\{v\in C_P[T,\infty): \frac{1}{2}P(t)\leq v(t)\leq
P(t),\;\; t\geq T\big\},
\end{equation}
where $C_P[T,\infty)$ denotes the Banach space of continuous
functions $v(t)$ on $[T,\infty)$ such that
$\|v\|_P=\sup\{|v(t)|/P(t): t\geq T\}<\infty$. Consider the
integral operator $G$ given by
\begin{equation}\label{3.14}
Gv(t)=P(t)-\int_T^t q(s)v(s)^2ds,\quad t\geq T,
\end{equation}
and let it act on $\mathcal{V}$. If $v\in \mathcal{V}$, then since
$$
\int_T^t q(s)v(s)^2ds\leq \int_T^t q(s)P(s)^2ds\leq P(t)\int_T^t
q(s)P(s)ds \leq \frac{1}{4}P(t),
$$
for $t\geq T$, we obtain $P(t)/2\leq Gv(t)\leq P(t)$ on
$[T,\infty)$. This shows that $G$ maps $\mathcal{V}$ into itself.
If $v_1,\;v_2\in\mathcal{V}$, then from the inequalities
$$
\frac{|Gv_1(t)-Gv_2(t)|}{P(t)}\leq \frac{1}{P(t)}\int_T^t
q(s)|v_1(s)^2-v_2(s)^2|ds
$$
$$
\leq 2\int_T^t q(s)P(s){\cdot}\frac{|v_1(s)-v_2(s)|}{P(s)}ds \leq
\frac{1}{2}\|v_1-v_2\|_P,
$$
we find that
$$
\|Gv_1-Gv_2\|_P\leq \frac{1}{2}\|v_1-v_2\|_P,
$$
that is, $G$ is a contraction on $\mathcal{V}$. Consequently, $G$
has a unique fixed point $v\in\mathcal{V}$, which satisfies
\begin{equation}\label{3.15}
v(t)=P(t)-\int_T^t q(s)v(s)^2ds,\quad t\geq T,
\end{equation}
and hence the Riccati equation (R2) on $[T,\infty)$. We claim that
\eqref{3.15} implies $v(t)\sim P(t)$ as $t\to\infty$. This follows
from \eqref{3.15} if it is confirmed that
\begin{equation}\label{3.16}
\lim_{t\to\infty}\frac{1}{P(t)}\int_T^t q(s)P(s)^2ds=0.
\end{equation}
Let $\varepsilon>0$ be given arbitrarily. Because of
\eqref{equation3.2} there exists $t_{\varepsilon}>T$ such that
\begin{equation}\label{3.17}
\int_{t_{\varepsilon}}^{\infty}q(s)P(s)ds<\frac{\varepsilon}{2}.
\end{equation}
Let this $t_{\varepsilon}$ be fixed and choose
$T_{\varepsilon}>t_{\varepsilon}$ so that
\begin{equation}\label{3.18}
\frac{1}{P(t)}\int_T^{t_{\varepsilon}}q(s)P(s)^2ds<\frac{\varepsilon}{2},\quad
t\geq T_{\varepsilon}.
\end{equation}
Then, using \eqref{3.17} and \eqref{3.18} we see that if
$t>T_{\varepsilon}$, then
$$
\frac{1}{P(t)}\int_T^t
q(s)P(s)^2ds=\frac{1}{P(t)}\Bigl(\int_T^{t_{\varepsilon}}
+\int_{t_{\varepsilon}}^t \Bigr)q(s)P(s)^2 ds
$$
$$
\leq \frac{\varepsilon}{2}+\int_{t_{\varepsilon}}^t q(s)P(s)ds
<\frac{\varepsilon}{2}+\frac{\varepsilon}{2}=\varepsilon.
$$
This clearly implies \eqref{3.16}, which guarantees the asymptotic
equivalence of the solution $v(t)$ of (R2) obtained above and the
function $P(t)$ as $t\to\infty$. Let us now construct a
nonoscillatory solution $x_2(t)$ of (A)
\begin{equation}\label{3.19}
x_2(t)=v(t)\exp\Bigl(-\int_t^{\infty}q(s)v(s)ds\Bigr),\quad t\geq
T,
\end{equation}
according to the reproducing formula \eqref{2.9}. Its
quasi-derivative is given by
\begin{equation}\label{3.20}
Dx_2(t)=\exp\Bigl(-\int_t^{\infty}q(s)v(s)ds\Bigr),\quad t\geq T.
\end{equation}
Since $x_2(t)\sim v(t)\sim P(t)$ and $Dx_2(t)\sim 1$ as
$t\to\infty$, $x_2(t)$ is a positive moderate solution of the type
I(i) and hence a nonprincipal solution of equation (A).

Note that by putting $u(t)=1/v(t)$, \eqref{3.19} and \eqref{3.20}
are rewritten as
$$
x_2(t)=\frac{1}{u(t)}\exp\Bigl(-\int_t^{\infty}\frac{q(s)}{u(s)}ds\Bigr)\quad
\textrm{and}\quad
Dx_2(t)=\exp\Bigl(-\int_t^{\infty}\frac{q(s)}{u(s)}ds\Bigr).
$$
Since $u(t)$ is a solution of (R1), this allows us to assert that
a solution $x_2(t)$ of the type (I)(i) of (A) can also be
reproduced from a global solution of the first Riccati equation
(R1).

Summarizing what are discussed above, we obtain the following
theorem which characterizes the structure of the totality of
moderate solutions of (A) for the case $I_p=\infty$.

\begin{theorem}\label{theorem3.1}\; {\it Assume that $I_p=\infty$. All
solutions of equation} {\rm (A)} {\it are moderate if and only if
the condition}
$$
\int_a^{\infty}q(t)P(t)dt<\infty\quad
\Bigl(or\;equivalently,\;\;\int_a^{\infty}\frac{\rho(t)}{p(t)}dt<\infty
\Bigr)
$$
{\it is satisfied. In this case} {\rm (A)} {\it has a moderate
basis consisting of solutions $x_1(t)$ and $x_2(t)$ which are
reproduced from global solutions of} {\rm (R1)} ({\it or} {\rm
(R2)}) {\it and satisfy, as $t\to\infty$,}
\begin{equation}\label{3.21}
x_1(t)\sim 1,\quad Dx_1(t)\sim -\rho(t),
\end{equation}
\begin{equation}\label{3.22}
x_2(t)\sim P(t),\quad Dx_2(t) \sim 1.
\end{equation}
\end{theorem}

Let us now turn our attention to equation (A) with $q(t)$
satisfying $I_q=\infty$. Use is made of the notation
\begin{equation}\label{3.23}
Q(t)=\int_a^t q(s)ds,\quad t\geq a.
\end{equation}

The situation in which (A) has moderate solutions of the types
II(i) and II(iii) can be characterized as the following theorem
shows.

\begin{theorem}\label{theorem3.2}\;{\it Assume that $I_q=\infty$. All
solutions of equation} \rm (A) {\it are moderate if and only if
the condition}
\begin{equation}\label{3.24}
\int_a^{\infty}q(t)\pi(t)dt<\infty\quad \Bigl(or\;equivalently,
\int_a^{\infty}\frac{Q(t)}{p(t)}dt<\infty \Bigr),
\end{equation}
{\it is satisfied. In this case} (A) {\it has a moderate basis
consisting of solutions $x_1(t)$ and $x_2(t)$ which are reproduced
from global solutions of} (R1) ({\it or} (R2)) {\it and satisfy,
as $t\to\infty$,}
\begin{equation}\label{3.25}
x_1(t)\sim \pi(t),\quad Dx_1(t)\sim -1\,
\end{equation}
\begin{equation}\label{3.26}
x_2(t)\sim 1,\quad Dx_2(t)\sim Q(t).
\end{equation}
\end{theorem}

\medskip

\begin{proof}\; (The "only if" part)\; Suppose that (A) has a positive
solution $x_1(t)$ of the type II(i) on $[T,\infty)$. It is bounded
since $x_1(\infty)=c$ for some $c>0$. Integrating (A) twice from
$T$ to $t$ one obtains
$$
x_1(t)=x_1(T)+\int_T^t \frac{Dx_1(T)}{p(s)}ds +\int_T^t
\frac{1}{p(s)}\int_T^s q(r)x_1(r)drds,\quad t\geq T,
$$
from which, in view of the boundedness of $x_1(t)$, it follows
that $I_p<\infty$ and
$$
\int_T^{\infty} \frac{1}{p(s)}\int_T^s q(r)drds<\infty.
$$
This clearly implies $\int_a^{\infty}Q(s)/p(s)ds<\infty$.

Let (A) possess a positive solution $x_2(t)$ of the type II(iii)
on $[T,\infty)$. Note that $Dx_2(\infty)=-d$ for some $d>0$ and
this implies
\begin{equation}\label{3.27}
\lim_{t\to\infty}\frac{x_2(t)}{\pi(t)}=d.
\end{equation}
Integrating (A) on $[T,\infty)$ one gets
$\int_T^{\infty}q(s)x_2(s)ds=-d-D(T)<\infty$, which combined with
\eqref{3.27} shows that $\int_T^{\infty}q(s)\pi(s)ds<\infty$.

(The "if" part)\; Assume that \eqref{3.24} holds. We first
construct a principal solution of (A) by using a solution of the
second Riccati equation (R2). Choose $T>a$ so that
\begin{equation}\label{3.28}
\int_T^{\infty}q(s)\pi(s)ds\leq \frac{1}{4},
\end{equation}
and look for a solution of the integral equation
\begin{equation}\label{3.29}
v(t)=-\pi(t)+\int_t^{\infty}q(s)v(s)^2ds,\quad t\geq T,
\end{equation}
lying in the set
\begin{equation}\label{3.30}
\mathcal{V}=\bigl\{v\in C_0[T,\infty):-\pi(t)\leq v(t)\leq
-\frac{1}{2}\pi(t),\;\; t\geq T \bigr\}.
\end{equation}
Consider the integral operator $G$ given by
\begin{equation}\label{3.31}
Gv(t)=-\pi(t)+\int_t^{\infty}q(s)v(s)^2ds,\quad t\geq T.
\end{equation}
Let $v\in \mathcal{V}$. Since
$$
\int_t^{\infty}q(s)v(s)^2ds\leq \int_t^{\infty}q(s)\pi(s)^2ds \leq
\pi(t)\int_t^{\infty}q(s)\pi(s)ds\leq \frac{1}{4}\pi(t),\quad
t\geq T,
$$
we have $-\pi(t)\leq Gv(t)\leq -\pi(t)/2$, $t\geq T$, which shows
that $G$ is a self-map of $\mathcal{V}$. If
$v_1,\;v_2\in\mathcal{V}$, then,
$$
|Gv_1(t)-Gv_2(t)|\leq \int_t^{\infty}q(s)|v_1(s)^2-v_2(s)^2|ds
\leq 2\int_t^{\infty}q(s)\pi(s)|v_1(s)-v_2(s)|ds
$$
$$
\leq 2\int_t^{\infty}q(s)\pi(s)ds{\cdot}\|v_1-v_2\|_0\leq
\frac{1}{2}\|v_1-v_2\|_0, \quad t\geq T,
$$
which implies that
$$
\|Gv_1-Gv_2\|_0\leq \frac{1}{2}\|v_1-v_2\|_0.
$$
This means that $G$ is a contraction of $\mathcal{V}$.
Consequently, $G$ has a unique fixed point $v\in \mathcal{V}$,
which satisfies \eqref{3.29} and hence gives a solution of (R2) on
$[T,\infty)$. Note that \eqref{3.29} and \eqref{3.30} guarantee
that $-v(t)\sim \pi(t)$ as $t\to\infty$. We now use the second
reproducing formula of \eqref{2.9} to form the function
\begin{equation}\label{3.32}
x_1(t)=-v(t)\exp\Bigl(-\int_t^{\infty}q(s)v(s)ds\Bigr), \quad
t\geq T.
\end{equation}
Then, $x_1(t)$ is a positive solution of (A) on $[T,\infty)$ and
satisfies $x_1(t)\sim \pi(t)$ as $t\to\infty$. It is easy to see
that $Dx_1(t)\sim -1$ as $t\to\infty$.

Put $u(t)=1/v(t)$. Then, $u(t)$ is a solution of (R1) and (3.32)
is rewritten as
$$
x_1(t)=-\frac{1}{u(t)}\exp\Bigl(-\int_t^{\infty}\frac{q(s)}{u(s)}ds\Bigr).
$$
So we may assert that $x_1(t)$ can be reproduced from a global
solution of (R1).

To construct a nonprincipal solution of (A) under condition
\eqref{3.24} use is made of the first Riccati equation (R1).
Choose $T>a$ so that
\begin{equation}\label{3.33}
\int_T^{\infty}\frac{Q(s)}{p(s)}ds\leq \frac{1}{4},
\end{equation}
and define the integral operator $F$ and the set $\mathcal{U}$ by
\begin{equation}\label{3.34}
Fu(t)=Q(t)-\int_T^t \frac{u(s)^2}{p(s)}ds,\quad t\geq T,
\end{equation}
and
\begin{equation}\label{3.35}
\mathcal{U}=\bigl\{u\in C_Q[T,\infty):\frac{1}{2}Q(t)\leq u(t)\leq
Q(t),\;\; t\geq T \bigr\},
\end{equation}
respectively. Here $C_Q[T,\infty)$ denotes the Banach space of
continuous functions $u(t)$ on $[T,\infty)$ satisfying
$\|u\|_Q=\sup\{|u(t)|/Q(t):t\geq T\}<\infty$. Let $u\in
\mathcal{U}$. Since by \eqref{3.33}
$$
\int_T^t \frac{u(s)^2}{p(s)}ds\leq \int_T^t \frac{Q(s)^2}{p(s)}ds
\leq Q(t)\int_T^t \frac{Q(s)}{p(s)}ds\leq \frac{1}{4}Q(t),\quad
t\geq T,
$$
one sees that $Q(t)\geq Fu(t)\geq Q(t)/2$, $t\geq T$. This shows
that $Fu\in \mathcal{U}$, implying that $F$ maps $\mathcal{U}$
into itself. Let $u_1,\;u_2\in \mathcal{U}$. Then, using
\eqref{3.32} again, we obtain
$$
\frac{|Fu_1(t)-Fu_2(t)|}{Q(t)}\leq \frac{1}{Q(t)}\int_T^t
\frac{u_1(s)+u_2(s)}{p(s)}|u_1(s)-u_2(s)|ds
$$
$$
\leq 2\int_T^t
\frac{Q(s)}{p(s)}{\cdot}\frac{|u_1(s)-u_2(s)|}{Q(s)}ds \leq
\frac{1}{2}\|u_1-u_2\|_Q,\quad t\geq T,
$$
which implies that
$$
\|Fu_1-Fu_2\|_Q \leq \frac{1}{2}\|u_1-u_2\|_Q.
$$
Therefore, $F$ has a unique fixed point $u\in \mathcal{U}$ which
satisfies
$$
u(t)=Q(t)-\int_T^t \frac{u(s)^2}{p(s)}ds,\quad t\geq T,
$$
and hence gives a global solution of the Riccati equation (R1) on
$[T,\infty)$. To obtain a positive solution of the type II(i) of
(A) it suffices to construct the function
\begin{equation}\label{3.36}
x_2(t)=\exp\Bigl(-\int_t^{\infty} \frac{u(s)}{p(s)}ds\Bigr),\quad
t\geq T,
\end{equation}
according to the second reproducing formula in \eqref{2.5}. It is
clear that $x_2(t)\sim 1$ as $t\to\infty$. Its quasi-derivative is
given by $Dx_2(t)=u(t)\exp(-\int_t^{\infty}u(s)/p(s)ds)$. Since
$$
\lim_{t\to\infty}\frac{1}{Q(t)}\int_T^t \frac{u(s)^2}{p(s)}ds=0,
$$
(see the second half of the proof of Theorem 3.1), it holds that
$u(t)\sim Q(t)$ as $t\to\infty$. This fact can be used to confirm
that $Dx_2(t)\sim Q(t)$ as $t\to\infty$. This completes the proof.
\end{proof}

It remains to analyze the asymptotic behavior of solutions of (A)
with $p(t)$ and $q(t)$ satisfying $I_p<\infty\wedge I_q<\infty$.
We know that in this case all solutions of (A) are moderate and
their terminal states are classified into the three types III(i),
III(ii) and III(iii). It will be shown that all of these types of
solutions of (A) can be reproduced from suitable global solutions
of the Riccati equations (R1) and (R2).

Let $x(t)$ be a solution of the type III(i) on $[T,\infty)$. We
may assume that $x(t)>0$ and $Dx(t)>0$ for $t\geq T$, so that
$x(\infty)=c$ and $Dx(\infty)=d$ for some constants $c>0$ and
$d>0$. To reproduce this solution from a solution of the Riccati
equation (R1) we proceed as follows. Let $\omega=d/c$. Choose
$T>a$ so that
$$
\rho(T)\leq \frac{\omega}{4}\quad \textrm{and}\quad \pi(t)\leq
\frac{1}{9\omega},
$$
define the integral operator $F_1$ by
\begin{equation}\label{3.37}
F_1u(t)=\omega-\rho(t)+\int_T^{\infty}\frac{u(s)^2}{p(s)}ds,\quad
t\geq T,
\end{equation}
and let $F_1$ act on the set
\begin{equation}\label{3.38}
\mathcal{U}_1=\bigl\{u\in C_b[T,\infty): |u(t)-\omega|\leq
\frac{\omega}{2},\;\; t\geq T\bigr\}.
\end{equation}
As is readily shown that $F_1(\mathcal{U}_1)\subset \mathcal{U}_1$
and that if $u_1, u_2\in \mathcal{U}_1$, then
$$
\|F_1u_1-F_1u_2\|_b\leq \frac{1}{3}\|u_1-u_2\|_b.
$$
By the contraction principle there exists a unique $u_1\in
\mathcal{U}_1$ satisfying the integral equation
\begin{equation}\label{3.39}
u_1(t)=\omega-\rho(t)+\int_T^{\infty}\frac{u_1(s)^2}{p(s)}ds,\quad
t\geq T,
\end{equation}
and hence the Riccati equation (R1) on $[T,\infty)$. With this
$u_1(t)$ construct a function $x_1(t)$ by
\begin{equation}\label{3.40}
x_1(t)=c\exp\Bigl(-\int_t^{\infty}\frac{u_1(s)}{p(s)}ds\Bigr),\quad
t\geq T.
\end{equation}
It is clear that $x_1(t)$ is a solution of (A) such that
$x_1(\infty)=c$. This $x_1(t)$ is the desired solution of the type
III(i) of (A) since its quasi-derivative
\begin{equation}\label{3.41}
Dx_1(t)=cu_1(t)\exp\Bigl(-\int_t^{\infty}\frac{u_1(s)}{p(s)}ds\Bigr),\quad
t\geq T,
\end{equation}
satisfies $Dx_1(\infty)=c\omega=d$.

Let $x(t)$ be any type-III(ii) positive solution of (A) on
$[T,\infty)$. It satisfies $x(\infty)=0$ and $Dx(\infty)=-d$ for
some $d>0$. Our aim is to reproduce $x(t)$ from a solution of the
Riccati equation (R2). Let $T>a$ be large enough that
$\pi(T)\rho(T)\leq 1/4$, define the integral operator $G$ and the
set $\mathcal{V}$ by
\begin{equation}\label{3.42}
Gv(t)=-\pi(t)+\int_t^{\infty}q(s)v(s)^2ds,\quad t\geq T,
\end{equation}
and
\begin{equation}\label{3.43}
\mathcal{V}=\bigl\{v\in C_0[T,\infty): -\pi(t)\leq v(t)\leq
-\frac{1}{2}\pi(t),\;\; t\geq T\bigr\}.
\end{equation}
Show that $G$ is a self-map of $\mathcal{V}$ and that
$$
\|Gv_1-Gv_2\|_0\leq \frac{1}{2}\|v_1-v_2\|_0\quad
\textrm{for\;any}\;\; v_1, v_2\in \mathcal{V}.
$$
Let $v$ denote the unique fixed point of $G$ in $\mathcal{V}$.
Then, it is a solution of (R2) on $[T,\infty)$. With this $v(t)$
form a solution of (A)
\begin{equation}\label{3.44}
x_2(t)=-dv(t)\exp\Bigl(-\int_t^{\infty}q(s)v(s)ds\Bigr),\quad
t\geq T,
\end{equation}
according to the reproducing formula \eqref{2.11}. Clearly,
$x_2(t)$ is a positive solution of (A) on $[T,\infty)$ satisfying
$x_2(\infty)=0$. Since
\begin{equation}\label{3.45}
Dx_2(t)=-d\exp\Bigl(-\int_t^{\infty}q(s)v(s)ds\Bigr),\quad t\geq
T,
\end{equation}
we see that $Dx_2(\infty)=-d$. Therefore $x_2(t)$ is a desired
solution of the type III(ii).

Finally let $x(t)$ be any positive III(iii)-type solution on
$[T,\infty)$. There is a constant $c>0$ such that $x(\infty)=c$
and $Dx(\infty)=0$. Choose $T>a$ so that $\rho(T)\pi(T)\leq 1/4$.
Then, one proves that the integral operator $F_2$ defined by
\begin{equation}\label{3.46}
Fu_2(t)=-\rho(t)+\int_t^{\infty}\frac{u(s)^2}{p(s)}ds,\quad t\geq
T.
\end{equation}
is a contraction on the set
\begin{equation}\label{3.47}
\mathcal{U}_2=\{u\in C_0[T,\infty):-\rho(t)\leq u(t)\leq
-\frac{1}{2}\rho(t),\;\; t\geq T\}.
\end{equation}
Therefore $F_2$ has a unique fixed point $u_2\in \mathcal{U}_2$
which reproduces a solution of (A)
\begin{equation}\label{3.48}
x_3(t)=c\exp\Bigl(-\int_t^{\infty}\frac{u_2(s)}{p(s)}ds \Bigr),
\quad t\geq T,
\end{equation}
whose quasi-derivative is given by
\begin{equation}\label{3.49}
Dx_3(t)=cu_2(t)\exp\Bigl(-\int_t^{\infty}\frac{u_2(s)}{p(s)}ds
\Bigr), \quad t\geq T,
\end{equation}
This solution $x_3(t)$ is of the type III(iii) since
$x_3(\infty)=c$ and $Dx_3(\infty)=0$.

The above discussions are summarized in the following theorem.

\begin{theorem}\label{theorem3.3}\; {\it Assume that $I_p<\infty\wedge
I_q<\infty$. All solutions of} \rm (A) {\it are divided into the
three classes consisting of moderate solutions of the types}
III(i), III(ii) {\it and} III(iii). {\it These classes are
represented, respectively, by the solutions $x_1(t)$, $x_2(t)$ and
$x_3(t)$ of} \rm (A) {\it which are reproduced from global
solutions of} \rm (R1) ({\it or} \rm (R2)) {\it having the
specified asymptotic behavior as $t\to\infty$,}
\begin{equation}\label{3.50}
x_1(t)\sim 1,\quad Dx_1(t)\sim 1.
\end{equation}
\begin{equation}\label{3.51}
x_2(t)\sim \pi(t),\quad Dx_2(t)\sim -1,
\end{equation}
\begin{equation}\label{3.52}
x_3(t)\sim 1,\quad Dx_3(t)\sim \rho(t).
\end{equation}
\end{theorem}

\smallskip

\begin{remark}\label{remark3.2}\;\; Of the above three solutions, $x_2(t)$
is a principal solution, while $x_1(t)$ and $x_3(t)$ are
nonprincipal solutions. It seems natural to adopt
$\{x_2(t),x_3(t)\}$ as a basis for the solution space
$\mathcal{S}$(A) of (A) in the case $I_p<\infty\wedge I_q<\infty$.
\end{remark}

\medskip

\section{Extreme solutions of equation (A)}

\indent

By definition a solution $x(t)$ of (A) is {\it extreme} if it has
the terminal state
\begin{center}
(a)\;\;$(|x(\infty)|=0,\;|Dx(\infty)|=0)$\quad or \quad
(b)\;\;$(|x(\infty)|=\infty,\;|Dx(\infty)|=\infty)$.
\end{center}
For simplicity a solution $x(t)$ satisfying (a) or (b) is termed,
respectively, a {\it decaying} extreme solution or a {\it growing}
extreme solution of (A).

All that are known at this stage about the existence of extreme
solutions of (A) are listed in the following

\begin{theorem}\label{theorem4.1}\; {\rm (i)}\; {\it If $I_p<\infty\wedge I_q<\infty$, then} \rm (A) {\it has
no extreme solutions}.

{\rm (ii)}\; {\it If $I_p=\infty\wedge I_q=\infty$, then all
solutions of} \rm (A) {\it are extreme and there exist both
decaying and growing extreme solutions}.

{\rm (iii)}\; {\it Let $I_p=\infty\wedge I_q<\infty$. All
solutions of} \rm (A) {\it are extreme, and there exist both
decaying and growing extreme solutions if and only if}
\begin{equation}\label{4.1}
\int_a^{\infty}q(t)P(t)dt=\infty \quad \bigg( or\ equivalently,
\int_a^{\infty}\frac{\rho(t)}{p(t)}dt=\infty \bigg).
\end{equation}

{\rm (iv)}\; {\it Let $I_p<\infty\wedge I_q=\infty$. All solutions
of} \rm (A) {\it are extreme and there exist both decaying and
growing extreme solutions if and only if}
\begin{equation}\label{4.2}
\int_a^{\infty}q(t)\pi(t)dt=\infty \quad \bigg(or \ equivalently,
\  \int_a^{\infty}\frac{Q(t)}{p(t)}dt=\infty \bigg).
\end{equation}
\end{theorem}

This theorem has little substance. For example, the propositions
(iii) and (iv) automatically follow from Theorems 3.2 and 3.3
based on the fact that extreme solutions and moderate solutions
cannot coexist for (A), and no information is available about how
to construct and determine their asymptotic behavior of such
solutions. As far as we know, no serious asymptotic analysis seems
to have ever been made of extreme solutions of (A) in the existing
literature.

Unlike moderate solutions, it is very difficult to have a good
grasp of extreme solutions of (A) presumably because of lack of a
priori information about their precise asymptotic behaviors as
$t\to\infty$. However, we are able to indicate some special cases
of (A) for which the existence of extreme solutions can actually
be established with the help of the Riccati equations (R1) and
(R2). The first result concerns growing extreme solutions of (A)
satisfying \eqref{4.1}.

\medskip
\begin{theorem}\label{theorem4.2}\; {\it Assume that $I_p=\infty\wedge
I_q<\infty$. In addition to} \eqref{4.1} {\it suppose that there
is a constant $\gamma\in (0,1)$ such that}
\begin{equation}\label{4.3}
\int_a^t q(s)P(s)^2ds\leq \gamma P(t)\;\;\;for\;all\;large\;\;t,
\end{equation}
{\it Then, equation} \rm (A) {\it possesses a positive growing
extreme solution $x(t)$ satisfying $x(\infty)=Dx(\infty)=\infty$.}
\end{theorem}

\begin{proof}\; Choose $T>a$ so that \eqref{4.3} holds for $t\geq T$. Let
$\mathcal{V}$ denote the set
\begin{equation}\label{4.4}
\mathcal{V}=\{v\in C[T,\infty): (1-\gamma)P(t)\leq v(t)\leq
P(t),\;\;t\geq T\},
\end{equation}
which is a closed convex subset of the locally convex space
$C[T,\infty)$ with the topology of uniform convergence on compact
subintervals of $[T,\infty)$. Define the integral operator
$G:\mathcal{V}\to C[T,\infty)$ by
\begin{equation}\label{4.5}
Gv(t)=P(t)-\int_T^t q(s)v(s)^2ds,\quad t\geq T.
\end{equation}

It can be shown that $G$ is a continuous self-map of $\mathcal{V}$
sending $\mathcal{V}$ into a relatively compact subset of
$C[T,\infty)$.

(i)\; Since by \eqref{4.4} $v\in \mathcal{V}$ implies
$$
\int_T^t q(s)v(s)^2ds\leq \int_T^t q(s)P(s)^2ds\leq \gamma
P(t),\quad t\geq T,
$$
we see that $P(t)\geq Gv(t)\geq (1-\gamma)P(t)$ for $t\geq T$.
This implies $G(\mathcal{V})\subset \mathcal{V}$.

(ii)\; Let $\{v_n(t)\}$ be a sequence in $\mathcal{V}$ such that
$v_n(t)\to v(t)$ as $n\to\infty$ uniformly on compact subintervals
of $[T,\infty)$. Noting that $|v_n(t)^2-v(t)^2|\leq 2P(t)^2$,
$t\geq T$, and $v_n(t)-v(t)\to 0$ at every point $t\in [T,\infty)$
as $n\to\infty$, and using the Lebesgue dominated convergence
theorem we see that
$$
|Gv_n(t)-Gv(t)|\leq \int_T^t q(s)|v_n(s)^2-v(s)^2|ds\to 0,\quad
t\to\infty,
$$
uniformly on any compact subinterval of $[T,\infty)$. This
establishes the continuity of $G$.

(iii)\; To prove the relative compactness of $G(\mathcal{V})$ it
suffices to show that $G(\mathcal{V})$ is locally uniformly
bounded and locally equicontinuous on $[T,\infty)$. The local
uniform boundedness is a consequence of the inclusion
$G(\mathcal{V})\subset \mathcal{V}$, while the local
equicontinuity follows from the inequality
$$
|(Gv)'(t)|\leq \frac{1}{p(t)}+q(t)P(t)^2,\quad t\geq T,
$$
holding for all $v\in \mathcal{V}$.

Therefore, by the Schauder-Tychonoff fixed point theorem (cf.
Coppel [3]), $G$ has a fixed element $v\in \mathcal{V}$ which
satisfies
$$
v(t)=P(t)-\int_T^t q(s)v(s)^2ds,\quad t\geq T,
$$
and hence the Riccati equation (R2) on $[T,\infty)$. Using this
$v(t)$ and the formula \eqref{2.9}, we reproduce a solution of
equation (A)
\begin{equation}\label{4.6}
x(t)=v(t)\exp\Bigl(\int_T^t q(s)v(s)ds\Bigr),\quad t\geq T,
\end{equation}
which has the quasi-derivative
\begin{equation}\label{4.7}
Dx(t)=\exp\Bigl(\int_T^t q(s)v(s)ds\Bigr), \quad t\geq T.
\end{equation}
It is clear that $x(\infty)=Dx(\infty)=\infty$ since
$v(\infty)=\infty$ and the integral $\int_T^t q(s)v(s)ds$ diverges
as $t\to\infty$. This proves Theorem 4.2.
\end{proof}

A decaying extreme solution of (A) satisfying \eqref{4.1} can be
reproduced from a solution of the Riccati equation (R1) as is
described in the following theorem.

\medskip

\begin{theorem}\label{theorem4.3}\; {\it Assume that $I_p=\infty\wedge
I_q<\infty$. Suppose in addition to} \rm (4.1) {\it that there is
a constant $\delta\in (0,1)$ such that}
\begin{equation}\label{4.8}
\int_t^{\infty} \frac{\rho(s)^2}{p(s)}ds\leq \delta
\rho(t)\;\;\;for\;all\;large \;\;\;t.
\end{equation}
{\it Then, equation} \rm (A) {\it possesses a positive decaying
extreme solution $x(t)$ satisfying $x(\infty)=Dx(\infty)=0$.}
\end{theorem}

\medskip

\begin{proof}\; Choose $T>a$ so that \eqref{4.8} holds for
$t\geq T$, define the integral operator $F$ by
\begin{equation}\label{4.9}
Fu(t)=-\rho(t)+\int_t^{\infty}\frac{u(s)^2}{p(s)}ds,\quad t\geq T,
\end{equation}
and let it act on the set
\begin{equation}\label{4.10}
\mathcal{U}=\{u\in C[T,\infty): -\rho(t)\leq u(t)\leq
-(1-\delta)\rho(t),\;\; t\geq T\}.
\end{equation}
As in the proof of the preceding theorem it can be shown that $F$
is a continuous self-map of $\mathcal{U}$ with the property that
$F(\mathcal{U})$ is a relatively compact subset of $C[T,\infty)$.
Therefore, the Schauder-Tychonoff theorem ensures the existence of
a function $u\in \mathcal{U}$ such that $u=Fu$, i.e.,
\begin{equation}\label{4.11}
u(t)=-\rho(t)+\int_t^{\infty}\frac{u(s)^2}{p(s)}ds,\quad t\geq T,
\end{equation}
so that $u(t)$ is a negative solution of the Riccati equation (R1)
on $[T,\infty)$. Using the reproducing formula \eqref{2.5} with
this $u(t)$ construct a positive solution of equation (A)
\begin{equation}\label{4.12}
x(t)=\exp\Bigl(\int_T^t \frac{u(s)}{p(s)}ds\Bigr),\quad t\geq T,
\end{equation}
whose quasi-derivative is given by
\begin{equation}\label{4.13}
Dx(t)=u(t)\exp\Bigl(\int_T^t \frac{u(s)}{p(s)}ds\Bigr),\quad t\geq
T.
\end{equation}
Since by \eqref{4.1}
$$
\int_T^t \frac{u(s)}{p(s)}ds\leq -(1-\delta)\int_T^t
\frac{\rho(s)}{p(s)}ds \to -\infty,\quad t\to\infty,
$$
it follows that $x(\infty)=Dx(\infty)=0$, which means that $x(t)$
is a decaying extreme solution of (A). This proves Theorem 4.3.
\end{proof}

\medskip

There exists a situation in which equation (A) certainly possesses
a pair of decaying and growing extreme solutions.

\medskip

\begin{corollary}\label{corollary4.1}\; {\it Assume that $I_p=\infty\wedge
I_q<\infty$. If conditions} \eqref{4.1}, \eqref{4.3} {\it and}
\eqref{4.8} {\it are satisfied, then equation} \rm (A) {\it has an
extreme solution basis $\{x_1(t),x_2(t)\}$ such that
$x_1(\infty)=Dx_1(\infty)=0$ and
$x_2(\infty)=Dx_2(\infty)=\infty$.}
\end{corollary}

Our next task is to  the existence of decaying extreme solutions
$x(t)$ of for equation (A) satisfying \eqref{4.2}.

\smallskip

\begin{theorem}\label{theorem4.4}\; {\it Assume that $I_p<\infty\wedge
I_q=\infty$. Suppose in addition to} \eqref{4.2} {\it that that
there is a constant $\gamma\in (0,1)$ such that}
\begin{equation}\label{4.14}
\int_t^{\infty} q(s)\pi(s)^2 ds\leq \gamma
\pi(t)\;\;\;for\;all\;large\;\;t.
\end{equation}
{\it Then, equation} \rm (A) {\it possesses a positive decaying
extreme solution $x(t)$ satisfying $x(\infty)=Dx(\infty)=0$.}
\end{theorem}
\vspace*{0.3cm}

\begin{theorem}\label{theorem4.5}\; {\it Assume that $I_p<\infty\wedge
I_q=\infty$. Suppose in addition to} \eqref{4.2} {\it that there
is a constant $\delta\in (0,1)$ such that}
\begin{equation}\label{4.15}
\int_a^t \frac{Q(s)^2}{p(s)}ds\leq \delta
Q(t)\;\;\;for\;all\;large\;\;t.
\end{equation}
{\it Then, equation} \rm (A) {\it possesses a positive growing
extreme solution $x(t)$ satisfying $x(\infty)=Dx(\infty)=\infty$.}
\end{theorem}

\smallskip

It will suffice to outline the proof of the above theorems. To
prove Theorem 4.5 choose $T>a$ so that \eqref{4.14} holds for
$t\geq T$, define the operator
$$
Gv(t)=-\pi(t)+\int_t^{\infty}q(s)v(s)^2ds,\quad t\geq T,
$$
and let $G$ act on the set
$$
\mathcal{V}=\{v\in C[T,\infty): -\pi(t)\leq v(t)\leq
-(1-\gamma)\pi(t),\;t\geq T\}.
$$
It can be shown routinely that $G$ is a self-map of $\mathcal{V}$,
that $G$ is a continuous map and that $G(\mathcal{V})$ is
relatively compact in $C[T,\infty)$. Therefore, by the
Schauder-Tychonoff theorem $G$ has a fixed point $v\in
\mathcal{V}$ which gives a solution of the Riccati equation (R2)
on $[T,\infty)$. With this $v(t)$ form a function
$$
x(t)=-v(t)\exp\Bigl(\int_T^t q(s)v(s)ds\Bigr),\quad t\geq T.
$$
Then, it is easily checked that $x(t)$ is a positive extreme
solution of (A) satisfying $x(\infty)=0$ and $Dx(\infty)=0$.

To prove Theorem 4.6 choose $T>a$ so that \eqref{4.15} holds for
$t\geq T$ and consider the set
$$
\mathcal{U}=\{u\in C[T,\infty): (1-\delta)Q(t)\leq u(t)\leq
Q(t),\;\;t\geq T\}.
$$
Then, letting the integral operator
$$
Fu(t)=Q(t)-\int_T^t \frac{u(s)^2}{p(s)}ds,\quad t\geq T,
$$
act on $\mathcal{U}$, one can verify that $F$ is a continuous
self-map of $\mathcal{U}$ such that $F(\mathcal{U})$ is a
relatively compact subset of $C[T,\infty)$. Therefore, $F$ has a
fixed point $u\in \mathcal{U}$ by the Schauder-Tychonoff theorem.
Clearly, $u(t)$ is a solution of the Riccati equation (R1) on
$[T,\infty)$. With this $u(t)$ form a positive function by
$$
x(t)=\exp\Bigl(\int_T^t \frac{u(s)}{p(s)}ds\Bigr),\quad t\geq T.
$$
Then, one easily sees that it is an extreme solution of (A)
satisfying $x(\infty)=Dx(\infty)=\infty$.

\medskip

\begin{corollary}\label{corollary4.2}\; {\it Assume that $I_p<\infty\wedge
I_q=\infty$. If conditions} \eqref{4.2}, \eqref{4.14} {\it and}
\eqref{4.15} {\it are satisfied, then equation} \rm (A) {\it has a
positive extreme solution basis $\{x_1(t),x_2(t)\}$ such that
$x_1(\infty)=Dx_1(\infty)=0$ and
$x_2(\infty)=Dx_2(\infty)=\infty$.}
\end{corollary}

\smallskip

The case of equation (A) with $I_p=\infty\wedge I_q=\infty$
remains to be examined. In this case all members $x(t)$ of
$\mathcal{S}$\textrm{(A)} are extreme solutions. A simple example
of such equations is
\begin{equation}\label{4.16}
(p(t)x')'=\frac{k^2}{p(t)}x,
\end{equation}
where $k>0$ is a constant. This is a special case of (A) with
$q(t)=k^2/p(t)$. Note that $I_q=k^2I_p$. The first Riccati
equation associated with \eqref{4.16} is
\begin{equation}\label{4.17}
u'=\frac{k^2-u^2}{p(t)}.
\end{equation}
Since \eqref{4.17} has the exact global solutions
\begin{equation}\label{4.18}
u_1(t)\equiv k,\quad u_2(t)\equiv -k,
\end{equation}
using the formula \eqref{2.5} we see that if $I_p=\infty$,
equation \eqref{4.16} has linearly independent extreme solutions
\begin{equation}\label{4.19}
x_1(t)=\exp(kP(t)),\qquad x_2(t)=\exp(-kP(t)), \quad t\geq a.
\end{equation}
One can use the second Riccati equation (R2) $v'=(1-k^2v^2)/p(t)$
to reach the same conclusion.

\smallskip

It is desirable to find a nontrivial class of equations of the
form (A) with $(p,q)$ satisfying $I_p=\infty\wedge I_q=\infty$
whose extreme solutions, either growing or decaying or both, can
be reproduced by way of the Riccati equations. However, we are
still far from solving this problem, and so we choose to close
this section by presenting an artificial method of making special
equations of the form (A) possessing exact extreme solutions that
are reproduced from exact global solutions of the associated
Riccati equation (R1) or (R2).

\smallskip

\begin{theorem}\label{theorem4.6}\;\; {\rm (i)}\;\; {\it Let $I_p=\infty$. If
$\varphi(t)$ is a positive $C^1$-function on $[a,\infty)$
satisfying $\varphi'(t)>0$ and $\varphi(\infty)=\infty$, then the
differential equation}
\begin{equation}\label{4.20}
(p(t)x')'=\Bigl(\frac{\varphi(t)^2}{p(t)}+\varphi'(t)\Bigr)x
\end{equation}
{\it has an extreme solution $x(t)$ such that
$x(\infty)=Dx(\infty)=\infty$.}

{\rm (ii)}\;\;{\it Let $I_p=\infty$. If $\Phi(t)$ is a positive
$C^1$-function on $[a,\infty)$ satisfying $\Phi'(t)<0$,
$\Phi(\infty)=0$ and}
\begin{equation}\label{4.21}
\int_a^{\infty}\frac{\Phi(t)^2}{p(t)}dt =\infty,
\end{equation}
{\it then the differential equation}
\begin{equation}\label{4.22}
(p(t)x')'=\Bigl(\frac{\Phi(t)^2}{p(t)}-\Phi'(t)\Bigr)x
\end{equation}
{\it has an extreme solution $x(t)$ such that
$x(\infty)=Dx(\infty)=0$.}
\end{theorem}

\vspace*{0.3cm}

\begin{theorem}\label{theorem4.7}\;\; {\rm (i)}\;\; {\it Let $I_p=\infty$. If
$\varphi(t)$ is a positive $C^1$-function on $[a,\infty)$
satisfying $\varphi'(t)<0$ and $\varphi(\infty)=0$, then the
differential equation}
\begin{equation}\label{4.23}
(p(t)x')'=\frac{1}{\varphi(t)^2}\Bigl(\frac{1}{p(t)}-\varphi'(t)\Bigr)x
\end{equation}
{\it has an extreme solution $x(t)$ such that
$x(\infty)=Dx(\infty)=\infty$.}

\rm (ii)\;\;{\it Let $I_p=\infty$. If $\Phi(t)$ is a positive
$C^1$-function on $[a,\infty)$ satisfying $\Phi'(t)>0$,
$\Phi(\infty)=\infty$ and}
\begin{equation}\label{4.24}
\int_a^{\infty}\frac{dt}{p(t)\Phi(t)^2}=\infty,
\end{equation}
{\it then the differential equation}
\begin{equation}\label{4.25}
(p(t)x')'=\frac{1}{\Phi(t)^2}\Bigl(\frac{1}{p(t)}+\Phi'(t)\Bigr)x
\end{equation}
{\it has an extreme solution $x(t)$ such that
$x(\infty)=Dx(\infty)=0$.}
\end{theorem}

\smallskip

To prove Theorem 4.8 it suffices to notice that $\varphi(t)$ (or
$-\Phi(t)$) is a solution of (R1) for \eqref{4.20} (or (R1) for
\eqref{4.22}) and to apply the formula \eqref{2.5} to obtain a
growing extreme solution $\exp(\int_a^t (\varphi(s)/p(s))ds)$ of
\eqref{4.20} and a decaying extreme solution $\exp(-\int_a^t
(\Phi(s)/p(s))ds)$ of \eqref{4.22}. In the statement (ii)
condition \eqref{4.21} is needed to assure that $q(t)$ satisfies
$I_q=\infty$. Theorem 4.9 can be proved analogously.

\medskip

\section{Examples}

\indent

We present some examples illustrating all the results obtained in
Sections 3 and 4. In the first example, $p$ is assumed to satisfy
$I_p=\infty$ and $P$ is given by $P(t)=\int_a^t ds/p(s)$.

\begin{example}\label{example5.1}\; Let $p(t)$ be a positive continuous
function on $[a,\infty)$ such that $I_p=\infty$ and consider the
linear differential equation
\begin{equation}\label{5.1}
(p(t)x')'=q(t)x,\quad q(t)=\frac{k}{p(t)P(t)^{\lambda}(\log
P(t))^{\mu}},
\end{equation}
where $k>0$, $\lambda>0$ and $\mu\in {\bf R}$ are constants.
Solutions of this equations are sought on intervals of the form
$[T,\infty)$, where $T\geq a$ should be such that $P(T)\geq e$.

Our attention is focused on the case $I_q<\infty$. This occurs if
and only if $(\lambda,\mu)$ is a member of the set
\begin{equation}\label{5.2}
\{(\lambda,\mu): \lambda>1, \mu\in {\bf R}\}\cup
\{(\lambda,\mu):\lambda=1,\mu>1\},
\end{equation}
which is abbreviated as $\{\lambda>1, \mu\in {\bf R}\}$ or
$\{\lambda=1, \mu>1\}$. In this case $\rho(t)$ has the asymptotic
property
$$
\rho(t)\sim \frac{k}{(\lambda-1)P(t)^{\lambda-1}(\log
P(t))^{\mu}},\quad t\to\infty, \quad \textrm{if}\quad \{\lambda>1,
\mu\in {\bf R}\},
$$
\begin{equation}\label{5.3}
\end{equation}
$$
\rho(t)\sim \frac{k}{(\mu-1)(\log P(t))^{\mu-1}},\quad
t\to\infty,\quad \textrm{if} \quad \{\lambda=1, \mu>1\}.
$$
Condition \eqref{3.5} holds for \eqref{5.1} if $(\lambda,\mu)$
satisfies
\begin{equation}\label{5.4}
\{\lambda>2, \mu\in {\bf R}\}\quad \textrm{or}\quad \{\lambda=2,
\mu>1\},
\end{equation}
in which case Theorem 3.2 ensures the existence of a moderate
basis $\{x_1(t), x_2(t)\}$ for \eqref{5.1} on some interval
$[T,\infty)$ which exhibit the asymptotic behavior
\begin{equation}\label{5.5}
x_1(t)\sim 1,\quad Dx_1(t)\sim \rho(t),\qquad x_2(t)\sim
P(t),\quad Dx_2(t)\sim 1, \quad t\to\infty.
\end{equation}
These solutions are reproduced by way of the Riccati equations
(R1) and (R2) associated with \eqref{5.1} as follows:
\begin{equation}\label{5.4}
x_1(t)=\exp\Bigl(-\int_t^{\infty}\frac{u(s)}{p(s)}ds\Bigr),\qquad
x_2(t)=v(t)\exp\Bigl(\int_T^t q(s)v(s)ds\Bigr),
\end{equation}
where $u(t)$ and $v(t)$ are, respectively, solutions of (R1) and
(R2) on some interval $[T,\infty)$ which satisfy $-\rho(t)\leq
u(t)\leq -\rho(t)/2$ and $P(t)/2\leq v(t)\leq P(t)\}$ there.

Turning to extreme solutions of \eqref{5.1}, we first note that
condition \eqref{4.1} holds for \eqref{5.1} if $(\lambda, \mu)$
satisfies
\begin{equation}\label{5.7}
\{\lambda=1, \mu>1\}\quad \textrm{or}\quad \{1<\lambda<2, \mu \in
{\bf R}\}\quad \textrm{or}\quad \{\lambda=2, \mu\leq 1\}.
\end{equation}
We then compute to see that
$$
\int_T^t q(s)P(s)^2ds \sim
\frac{kP(t)^{3-\lambda}}{(3-\lambda)(\log P(t))^{\mu}}, \quad
t\to\infty,
$$
from which it follows that
\begin{equation}\label{5.8}
\lim_{t\to\infty}\frac{1}{P(t)}\int_T^t q(s)P(s)^2ds=\begin{cases}
0& \text{for\;all\;$k>0$\; if $0<\mu\leq 1$}, \\
k& \text{for\;all\;$k<1$\; if $\mu=0$}.
\end{cases}
\end{equation}
Consequently, by Theorem 4.2 equation \eqref{5.1} possesses a
positive extreme solution $x(t)$ such that
$x(\infty)=Dx(\infty)=\infty$ for all $k$ if
$\{\lambda=2,\;0<\mu\leq 1\}$ and for all $k<1$ if
$\{\lambda=2,\;\mu=0\}$. In either case the solution $x(t)$ can be
represented in the form
\begin{equation}\label{5.9}
x(t)=v(t)\exp\Bigl(\int_T^t q(s)v(s)ds\Bigr),\quad t\geq T,
\end{equation}
in terms of a global solution $v(t)$ of (R2) for \eqref{5.1}
satisfying $(1-l)P(t)\leq v(t)\leq P(t)$ for $t\geq T$, where
$l>0$ is any constant such that $l<1$ if $0<\mu\leq 1$ and such
that $l<k$ if $\mu=0$.

To examine the applicability of Theorem 4.3 it should be noted
that the set of $(\lambda,\mu)$ in \eqref{5.7} for which
$\rho(t)^2/p(t)$ is integrable on $[a,\infty)$ is
\begin{equation}\label{5.10}
\bigl\{\lambda=2,\mu\leq 1\bigr\}\quad \textrm{or}\quad
\bigl\{2>\lambda>\frac{3}{2},\mu>\frac{1}{2}\bigr\}\quad
\textrm{or}\quad \bigl\{\lambda=\frac{3}{2},\mu\in {\bf R}\bigr\}.
\end{equation}
Then we compute:
$$
\bigl\{\lambda=\frac{3}{2},\mu\in {\bf R}\bigr\}\quad
\Longrightarrow \quad
\frac{1}{\rho(t)}\int_t^{\infty}\frac{\rho(s)^2}{p(s)}ds\sim
\frac{2k}{2\mu-1}\frac{P(t)^{\lambda-1}}{(\log P(t))^{\mu}}\to
\infty,\quad t\to\infty,
$$
$$
\bigl\{2>\lambda>\frac{3}{2},\mu>\frac{1}{2}\bigr\}\quad
\Longrightarrow \quad
\frac{1}{\rho(t)}\int_t^{\infty}\frac{\rho(s)^2}{p(s)}ds\sim
\frac{k}{(\lambda-1)(2\lambda-3)}\frac{P(t)^{2-\lambda}}{(\log
P(t))^{\mu}},\quad t\to\infty,
$$
and
$$
\bigl\{\lambda=2,\mu\leq 1\bigr\}\quad \Longrightarrow \quad
\frac{1}{\rho(t)}\int_t^{\infty}\frac{\rho(s)^2}{p(s)}ds\sim
\frac{k}{(\log P(t))^{\mu}}, \quad t\to\infty,
$$
which implies
\begin{equation}\label{5.11}
\lim_{t\to\infty}\frac{1}{\rho(t)}\int_t^{\infty}\frac{\rho(s)^2}{p(s)}ds
=\begin{cases}
0 & \text{for\;all\;$k$\;if\;$0<\mu \leq 1$},\\
k & \text{for\;all\;$k<1$\;if\;$\mu=0$}
\end{cases}
\end{equation}
This shows that Theorem 4.3 is applicable to \eqref{5.1} only in
the case where $\lambda=2$ and $0\leq \mu \leq 1$ and guarantees
the existence of a positive extreme solution $x(t)$ of equation
\eqref{5.1} such that $x(\infty)=Dx(\infty)=0$ for all $k$ if
$0<\mu\leq 1$ and for all $k<1$ if $\mu=0$. In either case the
solution $x(t)$ is represented in the form
\begin{equation}\label{5.12}
x(t)=\exp\Bigl(\int_T^t \frac{u(s)}{p(s)}ds\Bigr), \quad t\geq T,
\end{equation}
by using a global solution $u(t)$ of the Riccati equation (R1) for
\eqref{5.1} satisfying $-P(t)\leq u()\leq -(1-l)P(t)$ for $t\geq
T$, where $l>0$ is any constant such that $l<1$ if $\mu\leq 1$ and
such that $l<k$ if $\mu=0$.

Thus it is concluded that equation \eqref{5.1} possesses both
growing and decaying extreme solutions for all $k$ if $\lambda=2$
and $0<\mu\leq 1$ and for all $k<1$ if $\lambda=2$ and $\mu=0$.

\end{example}

\begin{remark}\label{remark5.1}\; The particular case
$\{\lambda=2,\;\mu=0\}$ of \eqref{5.1}, i.e.,
\begin{equation}\label{5.13}
(p(t)x')'=\frac{kx}{p(t)P(t)^2},\quad k>0,
\end{equation}
has a pair of exact extreme solutions
$\{P(t)^{\alpha_1},\;P(t)^{\alpha_2}\}$, where
\begin{equation}\label{5.14}
\alpha_1=\frac{1}{2}(1+\sqrt{1+4k}),\quad
\alpha_2=\frac{1}{2}(1-\sqrt{1+4k}).
\end{equation}
\end{remark}

In the second example, $p$ is assumed to satisfy $I_p<\infty$ and
$\pi$ is given by $\pi(t)=\int_t^{\infty}ds/p(s)$.

\smallskip

\begin{example}\label{example5.2}\; Let $p(t)$ be a positive continuous
function on $[a,\infty)$ such that $I_p<\infty$. Let $k>0$,
$\lambda$ and $\mu$ are constants and consider the linear
differential equation
\begin{equation}\label{5.15}
(p(t)x')'=q(t), \quad
q(t)=\frac{k}{p(t)}\Bigl(\frac{1}{\pi(t)}\Bigr)^{\lambda}
\Bigl(\log\Bigl(\frac{1}{\pi(t)}\Bigr)\Bigr)^{\mu},
\end{equation}
on $[T,\infty)$, where $T\geq a$ is chosen so that $\pi(T)\leq e$.

Note that $I_q=\infty$ if
\begin{equation}\label{5.16}
\{\lambda>1,\;\mu\in {\bf R}\}\quad \textrm{or}\quad
\{\lambda=1,\;\mu\geq-1\},
\end{equation}
and that the function $Q(t)$ has the asymptotic properties
$$
Q(t)\sim
\frac{k}{\lambda-1}\Bigl(\frac{1}{\pi(t)}\Bigr)^{\lambda-1}\Bigl(
\log\Bigl(\frac{1}{\pi(t)}\Bigr)\Bigr)^{\mu},\quad
t\to\infty,\quad \textrm{if} \quad \{\lambda>1,\;\;\mu\in {\bf
R}\},
$$
\begin{equation}\label{5.17}
\end{equation}
$$
Q(t)\sim
\frac{k}{\mu+1}\Bigl(\log\Bigl(\frac{1}{\pi(t)}\Bigr)\Bigr)^{\mu+1},\quad
t\to\infty,\quad \textrm{if}\quad \{\lambda=1,\;\;\mu>-1\},
$$
$$
Q(t)\sim k\log\log\Bigl(\frac{1}{\pi(t)}\Bigr),\quad t\to\infty,
\quad \textrm{if} \quad \{\lambda=1,\;\;\mu=-1\}.
$$

Since condition \eqref{3.24} is satisfied if
\begin{equation}\label{5.18}
\{\lambda<2,\;\;\mu\in {\bf R}\}\quad \textrm{or}\quad
\{\lambda=2,\;\;\mu\geq -1\},
\end{equation}
by Theorem 3.3 there exists a moderate basis $\{x_1(t),x_2(t)\}$
for equation \eqref{5.15} having the asymptotic behavior
\begin{equation}\label{5.19}
x_1(t)\sim \pi(t),\quad Dx_1(t)\sim -1,\qquad x_2(t)\sim
1,\;\;Dx_2(t)\sim Q(t), \quad t\to\infty.
\end{equation}
These solutions are reproduced by way of the Riccati equations
(R2) and (R1) for \eqref{5.13} as follows:
\begin{equation}\label{5.20}
x_1(t)=-v(t)\exp\Bigl(-\int_t^{\infty}q(s)v(s)ds\Bigr),\qquad
x_2(t)=\exp\Bigl(-\int_t^{\infty}\frac{u(s)}{p(s)}ds\Bigr),
\end{equation}
where $v(t)$ and $u(t)$ are, respectively, solutions of (R2) and
(R1) on some interval $[T,\infty)$ which satisfy $-\pi(t)\leq
v(t)\leq -\pi(t)/2$ and $Q(t)/2\leq u(t)\leq Q(t)\}$ there.

Turning to extreme solutions of \eqref{5.15}, we must notice that
condition \eqref{4.2} holds for \eqref{5.15} if $(\lambda, \mu)$
satisfies
\begin{equation}\label{5.21}
\{\lambda>2,\;\mu\in {\bf R}\} \quad \textrm{or}\quad
\{\lambda=2,\;\mu\geq-1\}.
\end{equation}
Then we find that
$$
\int_t^{\infty}q(s)\pi(s)^2ds\sim \begin{cases}
\displaystyle{\frac{2}{3-\lambda}\Bigl(\frac{1}{\pi(t)}\Bigr)^{\lambda-3}
\Bigl(\log\Bigl(\frac{1}{\pi(t)}\Bigr)\Bigr)^{\mu}}&
\text{if\;$\lambda<3$,\;$\mu\in {\bf R}$}, \\
\displaystyle{\frac{k}{-(\mu+1)}\Bigl(\log\Bigl(\frac{1}{\pi(t)}\Bigr)\Bigr)^{\mu+1}}
& \text{if\;$\lambda=3,\;\mu<-1$},
$$
\end{cases}
$$
as $t\to\infty$, from which it follows that
\begin{equation}\label{5.22}
\lim_{t\to\infty}\frac{1}{\pi(t)}\int_t^{\infty}q(s)\pi(s)^{2}ds=\begin{cases}
0 & \text{if\;$\{\lambda>2,\;\mu\in {\bf
R}\}$\;or\;$\{\lambda=2,\;-1\leq\mu<0\}$},
\\
k & \text{if\;$\{\lambda=2,\;\mu=0\}$}.
\end{cases}
\end{equation}
Therefore, it is concluded from Theorem 4.5 that equation
\eqref{5.15} possesses a decaying extreme solution for all $k>0$
if $\{\lambda>2\;\mu\in {\bf R}\}$ or $\{\lambda,\;-1\leq\mu<0\}$
and for all $k<1$ if $\{\lambda=2,\;\leq\mu=0\}$. In either case
the solution $x(t)$ is expressed in the form
\begin{equation}\label{5.23}
x(t)=-v(t)\exp\Bigl(\int_T^t q(s)v(s)ds\Bigr),\quad t\geq T,
\end{equation}
where $v(t)$ is a solution of (R2) for \eqref{5.15} satisfying
$-\pi(t)\leq v(t)\leq -\pi(t)/2$ for $t\geq T$, where $l>0$ is a
constant such that $l<1$ if $\{\lambda<2,\;\mu\in {\bf R}\}$ or
$\{\lambda=2,\;\mu<0\}$ and such that $l<k$ if
$\{\lambda=2,\;\mu=0\}$.

It should be noted that an extreme basis exists for equation
\eqref{5.15} for all $k$ if $\{\lambda,\;\mu<0\}$ and for all
$k<1$ if $\{\lambda=2,\;\mu=0\}$.
\end{example}

\begin{remark}\label{remark5.2}\; The simplest case $\{\lambda=2,\;\mu=0\}$
of \eqref{5.15}, i.e.,
\begin{equation}\label{5.24}
(p(t)x')'=\frac{kx}{p(t)P(t)^2},
\end{equation}
has two exact extreme solutions $P(t)^{\alpha_1}$ and
$P(t)^{\alpha_2}$, where $\alpha_1$ and $\alpha_2$ are given by
\eqref{5.14}.
\end{remark}

\medskip

Our next step is to reproduce growing extreme solutions of
\eqref{5.15} on the basis of Theorem 4.6. Necessary is the precise
information about the asymptotic behavior of $(1/Q(t))\int_T^t
Q(s)^2/p(s)ds$ as $t\to\infty$. It can be verified that if
$\{\lambda>3/2,\;\mu\in {\bf R}\}$, then
\begin{equation}\label{5.25}
\frac{1}{Q(t)}\int_T^t \frac{Q(s)^2}{p(s)}ds\sim
\frac{k}{(\lambda-1)(2\lambda-3)}\Bigl(\frac{1}{\pi(t)}\Bigr)^{\lambda-2}
\Bigl(\log\Bigl(\frac{1}{\pi(s)}\Bigr)\Bigr)^{2\mu},
\end{equation}
as $t\to\infty$, and that if $\{\lambda=3/2,\;\mu\in {\bf R}\}$,
then
\begin{equation}\label{5.26}
\frac{1}{Q(t)}\int_T^t \frac{Q(s)^2}{p(s)}ds\sim
\begin{cases}
\displaystyle{\frac{2k}{2\mu+1}\Bigl(\frac{1}{\pi(t)}\Bigr)^{-\frac{1}{2}}
\Bigl(\log\Bigl(\frac{1}{\pi(t)}\Bigr)\Bigr)^{\mu+1}}&
\text{if\;$\mu\neq -\frac{1}{2}$,} \\
\displaystyle{2k\Bigl(\frac{1}{\pi(t)}\Bigr)^{-\frac{1}{2}}
\Bigl(\log\Bigl(\frac{1}{\pi(t)}\Bigr)\Bigr)^{\frac{1}{2}}
\log\log\Bigl(\frac{1}{\pi(t)}\Bigr)} &
\text{if\;$\mu=-\frac{1}{2}$},
\end{cases}
\end{equation}
as $t\to\infty$. Since \eqref{5.25} and \eqref{5.26} implies that
\begin{equation}\label{5.27}
\lim_{t\to\infty}\frac{1}{Q(t)}\int_T^t
\frac{Q(s)^2}{p(s)}ds=\begin{cases}
\displaystyle{k} & \text{if\;$\{\lambda=2,\;\mu=0$\},} \\
\displaystyle{0} &
\text{if\;$\{\lambda=2,\;\mu<0\}$\;\;\textrm{or}\;
$\{\frac{3}{2}\leq \lambda<2,\;\mu\in {\bf R}\}$},
\end{cases}
\end{equation}
Theorem 4.6 shows that equation \eqref{5.15} possesses a growing
extreme solution for all $k<1$ if $\{\lambda=2\;\mu=0\}$ or
$\{3/2\leq \lambda<2,\;\mu\in {\bf R}\}$ and for all $k$ if
$\{\lambda=2,\;\mu=0\}$. In either case the solution $x(t)$ is
expressed in the form
\begin{equation}\label{5.28}
x(t)=\exp\Bigl(\int_T^t \frac{u(s)}{p(s)}ds\Bigr)ds,\quad t\geq T,
\end{equation}
in terms of a solution $u(t)$ of the Riccati equation (R1) for
\eqref{5.15} satisfying $(1-\delta)Q(t)\leq u(t)\leq Q(t)$ on
$[T,\infty)$ for some $\delta<k$ if $\{\lambda=2,\;\mu=0\}$ or for
some $\delta<1$ if $\{\lambda=2,\;\mu<0\}$ or $\{3/2\leq
\lambda<2,\;\mu\in {\bf R}\}$.

From the aforementioned we conclude that there exists an extreme
basis for equation \eqref{5.15} for all $k$ if
$\{\lambda=2\;\mu<0\}$ and for all $k<1$ if
$\{\lambda=2,\;\mu=0\}$.

\smallskip

\begin{remark}\label{remark5.3}\; The simplest case $\{\lambda=2,\;\mu=0\}$
of \eqref{5.15}, i.e.,
\begin{equation}\label{5.29}
(p(t)x')'=\frac{kx}{p(t)\pi(t)^2},
\end{equation}
has two exact extreme solutions $\pi(t)^{-\alpha_1}$ and
$\pi(t)^{-\alpha_2}$, where $\alpha_1$ and $\alpha_2$ are defined
by \eqref{5.14}.
\end{remark}
\medskip

\begin{example}\label{5.3}\; As an example illustrating Theorem
3.4
concerning the type-III moderate solutions of (A) we consider
equation \eqref{5.5} in which $q$ satisfies $I_q<\infty$. It is
clear that this holds if and only if $\{\lambda<1,\;\mu\in {\bf
R}\}$ or $\{\lambda=,\;\mu<-1\}$, in which case $\rho$ is given
asymptotically by
\begin{equation}\label{5.30}
\rho(t)\sim
\begin{cases}
\displaystyle{\frac{k}{1-\lambda}\Bigl(\frac{1}{\pi(t)}\Bigr)^{\lambda-1}
\Bigl(\log\Bigl(\frac{1}{\pi(t)}\Bigr)\Bigr)^{\mu}}&
\text{if\;$\{\lambda<1\;\mu\in {\bf R}\}$}, \\
\displaystyle{\frac{k}{-(\mu+1)}
\Bigl(\log\Bigl(\frac{1}{\pi(t)}\Bigr)\Bigr)^{\mu+1}}&
\text{if\;$\{\lambda=1,\;\mu<-1\}$}.
\end{cases}
\end{equation}
According to Theorem 3.4, there exist under these circumstances
three types of bounded moderate solutions $x_i(t)$, $i=1,2,3$, on
some interval $[T,\infty)$ such that
$$
x_1(t)\sim 1,\quad Dx_1(t)\sim 1,\qquad x_2(t)\sim \pi(t),\quad
Dx_2(t)\sim 1, \qquad x_3(t)\sim 1,\quad Dx_3(t)\sim -\rho(t),
$$
as $t\to\infty$. All of them are reproduced from suitable global
solutions of (R1) or (R2) associated with (5.15).
\end{example}

The final example is given to illustrate Theorem 4.9.

\begin{example}\label{example5.4}\;\;(i)\;\;The equation
\begin{equation}\label{5.31}
(e^{-t}x')'=(e^t + e^{3t})x
\end{equation}
is a special case of (4.20) with $p(t)=e^{-t}$ and
$\varphi(t)=e^{-t}$, and so from (i) of Theorem 4.9 we conclude
that \eqref{5.31} has a growing extreme solution
\begin{equation}\label{5.32}
x_1(t)=\exp\Bigl(\frac{1}{2}e^{2t}\Bigr), \quad t\geq 0.
\end{equation}
Using the formula \eqref{1.4} we then have a decaying extreme
solution of \eqref{5.31} given explicitly by
\begin{equation}\label{5.33}
x_2(t)=\exp\Bigl(\frac{1}{2}e^{2t}\Bigr)
\int_t^{\infty}\exp\bigl(s-e^{2s}\bigr)ds,\quad t\geq 0.
\end{equation}

(ii)\;\; Since the equation
\begin{equation}\label{5.34}
(e^{-3t}x')'=(e^{-t}+e^t)x
\end{equation}
is a special case of \eqref{4.22} with $p(t)=e^{-3t}$ and
$\Phi(t)=e^t$, by (ii) of Theorem 4.8 this equation has a decaying
extreme solution
\begin{equation}\label{5.35}
x_1(t)=\exp\Bigl(-\frac{1}{2}e^{2t}\Bigr),\quad t\geq 0.
\end{equation}
Using the formula \eqref{1.3} we obtain a growing extreme solution
of \eqref{5.34} expressed as
\begin{equation}\label{5.36}
x_2(t)=\exp\Bigl(-\frac{1}{2}e^{2t}\Bigr) \int_0^t
\exp\bigl(3s+e^{2s}\bigr)ds,\quad t\geq 0.
\end{equation}
\end{example}

\begin{remark}\label{remark5.3}\;\;(i)\;\;The solutions $x_1(t)$ and
$x_2(t)$ of \eqref{5.31} defined by \eqref{5.32} and \eqref{5.33}
satisfy $p(t)(x_1(t)x_2'(t)-x_1'(t)x_2(t))\equiv 1$, which can be
rewritten as
\begin{equation}\label{5.37}
\frac{x_2(t)}{Dx_2(t)}-\frac{x_1(t)}{Dx_1(t)}=\frac{1}{Dx_1(t)Dx_2(t)},\quad
t\geq 0.
\end{equation}
Note that $v_i(t)=x_i(t)/Dx_i(t)$, $i=1,2$, are solutions of the
Riccati equation (R2) that reproduce the solutions $x_i(t)$,
$i=1,2$, of \eqref{5.31}. Since $v_1(t)=e^{-t}$ is already known,
\eqref{5.37} can also be used to determine the second solution
$v_2(t)$ of (R2).

(ii)\;\; Likewise, the solutions $x_i(t)$, $i=1,2$, of
\eqref{5.34} given by \eqref{5.35} and \eqref{5.36} must satisfy
\begin{equation}\label{5.38}
\frac{x_2(t)}{Dx_2(t)}-\frac{x_1(t)}{Dx_1(t)}=-\frac{1}{Dx_1(t)Dx_2(t)},\quad
t\geq 0.
\end{equation}

\end{remark}

\bigskip

\end{document}